\documentclass[11pt]{article}
\usepackage{amssymb}
\usepackage{amsmath}
\usepackage{amsthm}
\usepackage{latexsym}
\usepackage{hyperref}
\usepackage{mathdots}
\usepackage[all]{xy}
\usepackage{stmaryrd}
\usepackage[normalem]{ulem}
\usepackage[capitalise]{cleveref}
\usepackage{tikz}
\usetikzlibrary{patterns}

\theoremstyle{definition}
\newtheorem{theo}{Theorem}[section]
\newtheorem{defi}[theo]{Definition}
\newtheorem{prop}[theo]{Proposition}
\newtheorem{nota}[theo]{Notation}

\newtheorem{cor}[theo]{Corollary}
\newtheorem{lemma}[theo]{Lemma}
\newtheorem{exa}[theo]{Example}
\newtheorem{rem}[theo]{Remark}
\numberwithin{equation}{section}
\newcommand{\DS}{\displaystyle}

\newcommand{\N}{{\mathbb N}}
\newcommand{\F}{{\mathbb F}}

\newcommand{\cA}{{\mathcal A}}
\newcommand{\cC}{{\mathcal C}}
\newcommand{\cD}{{\mathcal D}}

\newcommand{\cF}{{\mathcal F}}

\newcommand{\cH}{{\mathcal H}}
\newcommand{\cI}{{\mathcal I}}

\newcommand{\cM}{{\mathcal M}}

\newcommand{\cO}{{\mathcal O}}

\newcommand{\cB}{{\mathcal B}}
\newcommand{\cT}{{\mathcal T}}
\newcommand{\cL}{{\mathcal L}}

\newcommand{\cN}{{\mathcal N}}
\newcommand{\cU}{{\mathcal U}}
\newcommand{\cV}{{\mathcal V}}
\newcommand{\cW}{{\mathcal W}}
\newcommand{\cX}{{\mathcal X}}
\newcommand{\cY}{{\mathcal Y}}
\newcommand{\cZ}{{\mathcal Z}}

\newcommand{\qM}{$q$-matroid}
\newcommand{\T}{\mbox{$^{\sf T}$}}
\newcommand{\subspace}[1]{\mbox{$\langle{#1}\rangle$}}
\newcommand{\inner}[2]{\mbox{$\langle{#1}\!\mid\!{#2}\rangle$}}

\newcommand{\rk}{{\rm rk}\,}

\newcommand{\cl}{{\rm cl}}
\newcommand{\cyc}{{\rm cyc}}
\newcommand{\GL}{{\rm GL}}
\newcommand{\rowsp}{\mbox{\rm rs}}

\newcommand{\Hyp}{\mbox{$\text{Hyp}$}}

\newcounter{alp}
\newcounter{ara}
\newcounter{rom}
\newenvironment{romanlist}{\begin{list}{(\roman{rom})\hfill}{\usecounter{rom}
     \topsep-1ex \labelwidth.6cm \leftmargin.6cm \labelsep0cm
     \rightmargin0cm \parsep0ex \itemsep.5ex
     \partopsep0ex}}{\end{list}}
\newenvironment{alphalist}{\begin{list}{(\alph{alp})\hfill}{\usecounter{alp}
     \topsep-1.4ex \labelwidth.7cm \leftmargin.7cm \labelsep0cm
     \rightmargin0cm \parsep0ex \itemsep.5ex
     \partopsep-1.4ex}}{\end{list}}

\newenvironment{mylist2}{\begin{list}{(\arabic{ara})\hfill}{\usecounter{ara}
     \topsep-1ex \labelwidth.88cm \leftmargin.88cm \labelsep0cm
     \rightmargin0cm \parsep0ex \itemsep.5ex
     \partopsep-1ex}}{\end{list}}

\makeatletter
\let\@fnsymbol\@arabic
\makeatother
\begin{document}
\title{Decompositions of $q$-Matroids Using Cyclic Flats}
\author{Heide Gluesing-Luerssen\thanks{Department of Mathematics, University of Kentucky, Lexington KY 40506-0027, USA; heide.gl@uky.edu and benjamin.jany@uky.edu. HGL was partially supported by the grant \#422479 from the Simons Foundation.}\quad and Benjamin Jany\footnotemark[1]
}

\date{February 21, 2023}
\maketitle
	
\begin{abstract}
\noindent We study the direct sum of \qM{}s by way of their cyclic flats.
Using that the rank function of a \qM{} is fully determined by the cyclic flats and their ranks, we show
that the cyclic flats of the direct sum of two \qM{}s are exactly all the direct sums of the cyclic flats of the two summands.
This simplifies the rank function of the direct sum significantly.
A \qM{} is called irreducible if it cannot be written as a (non-trivial) direct sum.
We provide a characterization of irreducibility in terms of the cyclic flats and show that every \qM{} can be decomposed into a
direct sum of irreducible \qM{}s, which are unique up to equivalence.
\end{abstract}


\section{Introduction}\label{S-Intro}

A \qM{} is the $q$-analogue of a matroid. Its ground space is a finite-dimensional vector space over a finite field, and the rank function,
defined on the subspace lattice of that vector space, is the natural generalization of that of a matroid.
$q$-Matroids were introduced formally in 2018 in~\cite{JuPe18}, but actually appeared already in~\cite{Cra64}
and~\cite{BCJ17}.
These combinatorial objects have gained considerable attention in the coding theory community because of their close relation to rank-metric codes.
Indeed, vector rank-metric codes give rise to \qM{}s just like linear block codes induce matroids.
In the language of matroid theory the resulting ($q$-)matroids are simply the representable ones; see also \cref{D-ReprG} below.
Unsurprisingly, the induced ($q$-)matroids provide a lot of combinatorial information about the underlying codes.
For further details concerning rank-metric codes in relation to \qM{}s we refer to \cite{JuPe18,GJLR19,Shi19,GhJo20,GLJ22Gen} and the references therein.

Since their introduction in~\cite{JuPe18}, the theory of \qM{}s has made considerable progress.
Most notably, an extensive catalogue of cryptomorphisms has been derived in \cite{BCJ22}, and the notion of a direct sum of \qM{}s has
been introduced in \cite{CeJu21}.
While the cryptomorphisms form a (non-trivial) generalization of the corresponding ones for matroids, the same cannot be said about the direct sum.
None of the constructions of the direct sum known for matroids has a well-defined $q$-analogue.
The definition introduced in~\cite{CeJu21} is based on the rank function and makes use of the \qM{} union, a construction that is akin to associating a
matroid to a polymatroid (see \cref{T-Union}).
While this produces a \qM{} that has all the properties that one expects from a direct sum, it unfortunately does not appear to behave well
with respect to \qM{} structure: The rank function is somewhat unwieldy and there is no direct way of determining the
independent spaces (resp.\ circuits, bases, and flats) from the independent spaces (resp.\ circuits, bases, and flats) of the two summands.

Fortunately -- as we will show in this paper -- there is a \qM{} structure that does behave very well for the direct sum, namely the cyclic-flat structure.
A theory of cyclic flats has been first proposed in \cite{AlBy22}.
For studying the direct sum of \qM{}s, however, an alternative approach to cyclic flats is more beneficial.
It is presented in Sections~\ref{S-CycCore} and~\ref{S-CycFlats} of this paper.
In those sections and with methods different from those in~\cite{AlBy22}, it is shown that the cyclic flats together with their rank values fully determine the entire \qM{}.
In \cite{AlBy22} this is established by reconstructing the lattice of flats, which in turn leads to the \qM{}, while we determine the rank function of the entire \qM{} directly from the cyclic flats.
We will provide more details on the differences and similarities between~\cite{AlBy22} and this paper in \cref{R-AlBy22}.
In short, thanks to the different approaches, the results appear in different order and with different proofs.
Clearly both approaches have their merits.
For the purpose of this paper however,  our direct approach toward the rank function is better suited than the reconstruction of the lattice of flats -- simply because flats do not behave well when taking direct sums of \qM{}s.
In \cite[Sec.~3]{AlBy22}, the authors use their results to derive a cryptomorphism for \qM{}s based on cyclic flats and their rank values.

Equipped with a sufficient theory on cyclic flats, we turn to the direct sum of \qM{}s.
After recalling and deriving some basic properties in \cref{S-DirSum}, we will
show in \cref{S-DSProperties} that the cyclic flats of a direct sum $\cM_1\oplus\cM_2$ is the collection of the direct sums of all
cyclic flats of~$\cM_1$ and~$\cM_2$.
As a consequence,  the rank function of $\cM_1\oplus\cM_2$ is fully determined by the cyclic flats of~$\cM_1$ and~$\cM_2$ along with their rank values.
This leads to a highly efficient formula of the rank function.
Even more, this result allows us to study decompositions of \qM{}s into irreducible ones.
Naturally, a \qM{} is called irreducible if it cannot be written as a direct sum of \qM{}s with
nonzero ground spaces. This can be nicely characterized in terms of the cyclic flats.
Furthermore, we show that every \qM{} can be written as a direct sum of irreducible ones, whose summands are unique up to equivalence.

We wish to point out that the use of cyclic flats for studying direct sums and decomposing \qM{}s appears to be novel in the
sense that it does not have a blueprint in classical
matroid theory (as far as we know).
It is not surprising that such an approach does not exist for matroids because in that case each of the immediate structures
(rank, independent spaces, flats etc.) can be used to conveniently define the direct sum.
Even more, the notion of connectedness for matroids leads to an equivalence relation, which in turn
produces the decomposition of a matroid into its irreducible (i.e., connected) components.
No $q$-analogue of this equivalence relation that serves the same purpose has been found yet for \qM{}s; see for instance \cite[Sec.~8]{CeJu21}.

\textbf{Notation:}
Throughout, let $\F=\F_q$ and~$E$ be a finite-dimensional $\F$-vector space.
We use $\cL(E)$  to denote the lattice of all subspaces of~$E$.
The subspace generated by~$x_1,\ldots,x_t\in E$ is denoted by $\subspace{x_1,\ldots,x_t}$.
We write $U\leq V$ if $U$ is a subspace of~$V$ and $U\lneq V$ if $U$ is a proper subspace of~$V$.
Furthermore, we denote by $\Hyp(V)$ the set of all hyperplanes of~$V$.
These are the codimension-1 subspaces of~$V$ and must not be confused with hyperplanes in the ($q$-)matroid sense.
For a matrix $M\in\F^{k\times n}$ we denote by $\rowsp(M)\in\cL(\F^n)$ the row space of~$M$.
We let $e_1,\ldots,e_n$ be the standard basis vector in~$\F^n$.

\section{Basic Notions of $q$-Matroids}\label{S-Basics}
In this section we collect some facts about \qM{}s that will be crucial later on.

\begin{defi}\label{D-qMatroid}
A \emph{$q$-matroid with ground space~$E$} is a pair $\cM=(E,\rho)$, where $\rho: \cL(E)\longrightarrow\N_{\geq0}$ is a map satisfying
\begin{mylist2}
\item[(R1)\hfill] Dimension-Boundedness: $0\leq\rho(V)\leq \dim V$  for all $V\in\cL(E)$;
\item[(R2)\hfill] Monotonicity: $V\leq W\Longrightarrow \rho(V)\leq \rho(W)$  for all $V,W\in\cL(E)$;
\item[(R3)\hfill] Submodularity: $\rho(V+W)+\rho(V\cap W)\leq \rho(V)+\rho(W)$ for all $V,W\in\cL(E)$.
\end{mylist2}
$\rho(V)$ is called the \emph{rank of}~$V$ and $\rho(\cM):=\rho(E)$ the \emph{rank} of the \qM.
\end{defi}

Clearly, $\rho(V+\subspace{x})\in\{\rho(V),\,\rho(V)+1\}$ for all $V,\subspace{x}\in\cL(E)$. The following results will play a crucial role later on.

\begin{prop}[\mbox{\cite[Prop.~6]{JuPe18}, \cite[Lem.~3.2]{BCIJS23}}]\label{P-RhoProp}
Let $\cM=(E,\rho)$ be a \qM{} and $V,W\in\cL(E)$.
\begin{alphalist}
\item Suppose $\rho(V+\subspace{w})=\rho(V)$ for all $w\in W$.
         Then $\rho(V+W)=\rho(V)$.
\item Suppose $W\leq V$. Let $x\in E$ be such that $\rho(W+\subspace{x})=\rho(W)$. Then $\rho(V+\subspace{x})=\rho(V)$.
\end{alphalist}
\end{prop}

\begin{theo}[\mbox{\cite[Thm.~42]{JuPe18}}]\label{T-DualqM}
Let $\inner{\cdot}{\cdot}$ be a non-degenerate symmetric bilinear form (NSBF) on~$E$.
For $V\in\cL(E)$ define $V^\perp=\{w\in E\mid \inner{v}{w}=0\text{ for all }v\in V\}$.
Let $\cM=(E,\rho)$ be a \qM{} and set
\begin{equation}\label{e-rhodual}
    \rho^*(V)=\dim V+\rho(V^\perp)-\rho(E).
\end{equation}
Then $\cM^*=(E,\rho^*)$ is a \qM{}, called the \emph{dual} of~$\cM$ with respect to the chosen NSBF.
Furthermore, $\cM^{**}=\cM$, where $\cM^{**}=(\cM^*)^*$ is the bidual.
\end{theo}

The dual \qM{} depends on the choice of $\inner{\cdot}{\cdot}$, but different  NSBFs lead to equivalent dual
$q$-matroids in the following sense \cite[Thm.~2.8]{GLJ22Gen}.

\begin{defi}\label{D-Equiv}
Let $E_i,\,i=1,2,$ be $\F$-vector spaces of the same finite dimension and $\cM_i=(E_i,\rho_i)$ be \qM{}s.
We call $\cM_1$ and $\cM_2$ \emph{equivalent}, denoted by $\cM_1\approx\cM_2$, if there exists an $\F$-isomorphism
$\alpha: E_1\longrightarrow E_2$ such that
$\rho_2(\alpha(V))=\rho_1(V)$ for all $V\in\cL(E_1)$.
\end{defi}

\begin{defi}\label{D-Unif}
Let $0\leq k\leq \dim E$. The \emph{uniform \qM{} of rank~$k$ on~$E$} is defined as the \qM{} $\cU_k(E)=(E,\rho)$, where $\rho(V)=\min\{k,\dim V\}$
for all $V\in\cL(E)$. Moreover, $\cU_0(E)$ is called the \emph{trivial \qM} on~$E$ and $\cU_{\dim E}(E)$ is the \emph{free \qM} on~$E$.
We let $\cU_{0,n}$ and $\cU_{n,n}$ denote the trivial and free \qM{} on any $\F_q$-linear  ground space of
dimension~$n$ (hence they are defined only up to equivalence).
\end{defi}

Note that $\cU_{k}(E)^*=\cU_{\dim E-k}(E)$ for any NSBF on~$E$.
In particular, the dual of the free \qM{} is the zero \qM{} and vice versa.

\begin{defi}[\mbox{see for instance \cite[Def.~5.1 and Thm.~5.2]{GLJ22Gen}}]\label{D-RestrContr}
Let $\cM=(E,\rho)$ be a \qM{} and $X\in\cL(E)$.
The \emph{restriction of~$\cM$ to~$X$} and the \emph{contraction of~$X$ from~$\cM$}
are defined as the \qM{}s $\cM|_X=(X,\hat{\rho})$ with $\hat{\rho}(V)=\rho(V)$ for all $V\leq X$ and
$\cM/X=(E/X,\tilde{\rho})$ with $\tilde{\rho}(V/X)=\rho(V+X)-\rho(X)$ for all $V\in\cL(E)$, respectively.
\end{defi}

The following notions of $q$-matroids have been well studied; for details see \cite{JuPe18,BCJ22}.
Let $\cM=(E,\rho)$ be a $q$-matroid. A subspace $I\in\cL(E)$ is called \emph{independent} if $\rho(I)=\dim I$.
Every subspace of an independent space is independent and furthermore,
\begin{equation}\label{e-rhoVI}
    \rho(V)=\max\{\dim I\mid I\leq V,\,I\text{ independent}\}\ \text{ for all }\ V\in\cL(E).
\end{equation}
As a consequence, the collection of independent spaces fully determines the entire \qM{}~$\cM$.
If $V\in\cL(E)$ is not independent it is called \emph{dependent}.
A \emph{basis} of  a subspace~$V$ is an independent space~$I\leq V$ satisfying $\rho(I)=\rho(V)$.
A dependent space all of whose proper subspaces are independent is a \emph{circuit}.
A space is called \emph{open} if it is a sum of circuits (by definition, the empty sum is the zero space).
Thus  any open space $V\in\cL(E)$ is of the form
\begin{equation}\label{e-Vopen}
   V=\sum_{   \genfrac{}{}{0pt}{1}{C\leq V}{C\text{ circuit}     }}\hspace*{-.5em}C.
\end{equation}
A space $F\in\cL(E)$ is a \emph{flat} if it is maximal with respect to its rank value, that is, $\rho(F+\subspace{x})>\rho(F)$ for all $x\in E\setminus F$.
Clearly, $E$ is a flat.
A flat~$F$ is a \emph{hyperplane} if $F\neq E$ and~$E$ is the only flat properly containing~$F$.
We denote by $\cI(\cM)$  the collection of all independent spaces of~$\cM$ and similarly use
$\cD(\cM),\;\cC(\cM),\;\cO(\cM),\;\cF(\cM),\;\cH(\cM)$ for the collections of the other spaces.
Finally, we need the closure operator of~$\cM$ which is defined as
\[
   \cl:\cL(E)\longrightarrow\cL(E),\quad V\longmapsto\hspace*{-.5em}
   \sum_{   \genfrac{}{}{0pt}{1}{x\in E}{\rho(V+\langle x\rangle)=\rho(V)}     }\hspace*{-1em}\subspace{x}.
\]
Clearly, a subspace~$V$ is a flat if and only if $\cl(V)=V$.
Thanks to \cref{P-RhoProp}(a) we have
\begin{equation}\label{e-clrho}
  \rho(\cl(V))=\rho(V)\text{ for all }V\in\cL(E).
\end{equation}

We now list some fundamental properties of \qM{}s.
Part (a) and (b) are in \cite[Cor.~86, Cor.~79]{BCJ22}, and~(c) is a consequence of~(a) and~(b) together with \eqref{e-Vopen}.
Part~(d) is in \cite[Def.~10 and Thm.~44]{BCJ22}. The last parts about the closure function are in \cite[Def.~5 and Cor.~56]{BCJ22}.

\begin{theo}\label{T-Basics}
Let $\cM=(E,\rho)$ be a \qM{} and $\cM^*=(E,\rho^*)$ be its dual with respect to some NSBF $\inner{\cdot}{\cdot}$.
Let $V\in\cL(E)$. Then
\begin{alphalist}
\item $V\in\cO(\cM)\Longleftrightarrow V^\perp\in\cF(\cM^*)$.
\item $V\in\cC(\cM)\Longleftrightarrow V^\perp\in\cH(\cM^*)$.
\item If $V\in\cF(\cM)$, then ${\DS V=\hspace*{-1em}\bigcap_{\genfrac{}{}{0pt}{1}{F\leq H}{H\in\cH(\cM)} }\hspace*{-1em}H}$.
\item $F_1,F_2\in\cF(\cM)\Longrightarrow F_1\cap F_2\in\cF(\cM)\ $ and $\ O_1,O_2\in\cO(\cM)\Longrightarrow O_1+O_2\in\cO(\cM)$
\item The closure operator satisfies for all $V,W\in\cL(E)$
\[
   V\leq\cl(V),\quad V\leq W\Longrightarrow \cl(V)\leq\cl(W),\quad \cl(\cl(V))=\cl(V).
\]
\end{alphalist}
\end{theo}

A particularly nice class of \qM{}s are the representable ones.

\begin{defi}[\mbox{\cite[Sec.~5]{JuPe18}}]\label{D-ReprG}
Consider a field extension $\F_{q^m}$ of~$\F=\F_q$ and let $G\in\F_{q^m}^{k\times n}$ be a matrix of rank~$k$.
Define the map $\rho:\cL(\F^n)\longrightarrow\N_0$ via
\[
    \rho(V)=\rk (GY\T),\ \text{ where $Y\in\F^{y\times n}$ such that $V=\rowsp(Y)$}.
\]
Then~$\rho$ satisfies (R1)--(R3) and thus defines a \qM{} $\cM_G=(\F^n,\rho)$ of rank~$k$.
We call $\cM_G$ the \qM{} \emph{represented by~$G$}.
\end{defi}

\section{The Cyclic Core}\label{S-CycCore}
We define the cyclic core of a subspace.
Our definition (\cref{D-CycCore}) is a vector-based $q$-analogue of the classical cyclic core in \cite[p.~395]{FHHW18}.
More precisely,  our cyclic core consists of vectors in the ground space~$E$ (rather than subspaces) that behave in the desired way
with respect to hyperplanes, and it is not hard to see that the so defined set is indeed a subspace.
Making use of duality and the properties of the closure operator, we show that the cyclic core of a subspace~$V$ is the largest open space
contained in~$V$. This also shows immediately that the cyclic-core operator is the dual of the closure operator (\cref{C-cyccl}) and idempotent
(\cref{P-CycProp}).
Subspaces that coincide with their cyclic core are called cyclic spaces, and it turns out that these are exactly the open spaces.

%
As mentioned in the introduction, a theory of the cyclic core was first proposed in \cite{AlBy22}.
Here we develop an alternative approach, which is well suited for our study of the direct sum of \qM{}s in subsequent sections.
We briefly describe the differences and similarities.

\begin{rem}\label{R-AlBy22}
The main difference between~\cite{AlBy22} and our work is the starting point for developing the theory.
In~\cite{AlBy22}, the authors start with the definition of cyclic spaces.
This can easily be done without the need of a cyclic-core operator and is identical to what we call cyclic spaces.
Furthermore, the authors define two concepts that may be regarded as $q$-analogues of the classical cyclic core:
$\text{Cyc}(V)$ is defined as a collection of certain 1-dimensional subspaces (\cite[Def.~2.3]{AlBy22}) and the
cyclic-core operator $\text{cyc}(V)$ is the sum of the cyclic subspaces contained in~$V$ (\cite[Def.~2.8]{AlBy22}).
The relation between these two concepts is clarified in \cite[Prop.~2.18]{AlBy22}, which then also shows that our cyclic-core operator
coincides with the one in~\cite{AlBy22}.
Due to the different routes taken, the authors in \cite[Sec.~2.1]{AlBy22} establish properties of cyclic spaces in order to study their cyclic-core operator, while we establish properties of our cyclic-core operator and obtain those for cyclic spaces as a consequence.
Thus, while the results look almost identical, they arise in different order and their proofs differ in each single case.
\end{rem}

\begin{nota}\label{Nota}
Throughout, let $\inner{\cdot}{\cdot}$ be a fixed NSBF on~$E$, thus $V^\perp$ denotes the orthogonal space of~$V$
with respect to $\inner{\cdot}{\cdot}$.
Let $\cM=(E,\rho)$ be a \qM{} and $\cM^*=(E,\rho^*)$ be its dual with respect to $\inner{\cdot}{\cdot}$
We denote by $\cl(\,\cdot\,)$ and $\cl^*(\,\cdot\,)$ the closure operators of~$\cM$ and $\cM^*$, respectively.
\end{nota}

\begin{defi}\label{D-CycCore}
For $V\in\cL(E)$ we define the \emph{cyclic core} of $V$ as
\[
   \cyc(V)=\{x\in V\mid \rho(W)=\rho(V)\text{ for all $W\leq V$ such that $W+\subspace{x}=V$}\}.
\]
\end{defi}
Note that every subspace $W$ appearing in the definition has codimension at most~$1$ in~$V$.
Clearly, if~$V$ is independent then $\cyc(V)=0$. The converse is true as well as we will see in \cref{T-CycCoreOpen}.

\begin{prop}\label{P-cyclinear}
Let $V\in\cL(E)$. Then $\cyc(V)\in\cL(E)$, i.e., $\cyc(V)$ is a subspace of~$E$.
\end{prop}

\begin{proof}
Clearly $0\in\cyc(V)$ and $\cyc(V)$ is closed under scalar multiplication.
Let $x,y\in\cyc(V)$ and let $W\leq V$ be such that $W+\subspace{x+y}=V$.
We want to show that $\rho(W)=\rho(V)$.
Without loss of generality let $W\neq V$. Then $\dim W=\dim V-1$ and $x+y\not\in W$.
Thus we may assume that $x\not\in W$. Hence $W+\subspace{x}=V$, and using that~$x$ is in $\cyc(V)$ we conclude
$\rho(W)=\rho(V)$.
\end{proof}

We will show below that $\cyc(V)$ is the largest open space contained in~$V$.
We will do so by making use of duality and the closure operator.

\begin{lemma}\label{L-dualClos}
Let $V\in\cL(E)$ and $C_1,\ldots,C_t\in\cC(\cM)$ be the circuits of~$\cM$ contained in~$V$. Then
\[
   \cl^*(V^\perp)=\bigcap_{i=1}^t C_i^\perp.
\]
\end{lemma}

\begin{proof}
Since $\cl^*(V^\perp)$ is a flat in~$\cM^*$, \cref{T-Basics} implies
\[
  \cl^*(V^\perp)=\bigcap_{\genfrac{}{}{0pt}{1}{\cl^*(V^\perp)\leq H}{H\in\cH(\cM^*)}}\hspace*{-1em}H\quad
    =\bigcap_{\genfrac{}{}{0pt}{1}{\cl^*(V^\perp)\leq C^\perp}{C\in\cC(\cM)} }\hspace*{-1em}C^\perp\quad
    =\bigcap_{\genfrac{}{}{0pt}{1}{C\leq V}{C\in\cC(\cM)} }\hspace*{-.5em}C^\perp,
\]
where the last step follows from the fact that $C^\perp$ is a flat of~$\cM^*$, which implies the equivalences
$\cl^*(V^\perp)\leq C^\perp\Longleftrightarrow V^\perp\leq C^\perp\Longleftrightarrow C\leq V$.
\end{proof}

Now we obtain the following description of the cyclic core.
Using either the theory developed in \cite{AlBy22} or ours, it is easy to verify that it agrees with the one defined in \cite[Def.~2.8]{AlBy22}.

\begin{theo}\label{T-CycCoreOpen}
For $V\in\cL(E)$ we have
\[
   \cyc(V)=\sum_{\genfrac{}{}{0pt}{1}{C\leq V}{C\in\cC(\cM)} }\hspace*{-.5em}C.
\]
Thus $\cyc(V)$ is an open space and in fact the largest open space contained in~$V$.
As a consequence, $V$ is independent if and only if $\cyc(V)=0$.
\end{theo}

\begin{proof}
Let $C_1,\ldots,C_t\in\cC(\cM)$ be the circuits of~$\cM$ contained in~$V$. We have to show that $\cyc(V)=\sum_{i=1}^t C_i$.
From \cref{L-dualClos} we know that $\cl^*(V^\perp)=\cap_{i=1}^t C_i^\perp$, and thus \eqref{e-clrho} implies
\begin{equation}\label{e-rhostar}
    \rho^*(V^\perp)=\rho^*\Big(\bigcap_{i=1}^t C_i^\perp\Big).
\end{equation}
We now turn to the stated identity.
\\
``$\subseteq$'' Let $x\in\cyc(V)$. Let $W\leq V$ be such that $\dim W=\dim V-1$ and $W+\subspace{x}=V$. Then $\rho(W)=\rho(V)$ and thus
\[
  \rho^*(W^\perp)=\dim W^\perp+\rho(W)-\rho(E)=\dim (V^\perp)+1+\rho(V)-\rho(E)=\rho^*(V^\perp)+1.
\]
Thus $\rho^*(W^\perp)>\rho^*\Big(\bigcap_{i=1}^t C_i^\perp\Big)$, which in turn implies that $W^\perp$ is not a subspace of
$\bigcap_{i=1}^t C_i^\perp$.
Hence $\sum_{i=1}^t C_i\not\leq W$.
Since this is true for every $W\in\Hyp(V)$ not containing~$x$, we conclude that $x\in\sum_{i=1}^t C_i$.
\\
``$\supseteq$'' It suffices to show that each $C_i$ is in $\cyc(V)$. Let $x\in C_i$.
We show that $\rho(W)=\rho(V)$ for all $W\in\Hyp(V)$  such that $W+\subspace{x}=V$.
Choose such a subspace~$W$. Then clearly $x\not\in W$ and thus $C_i\not\leq W$.
Using the containment $C_i\leq V$ we obtain
$\cl^*(V^\perp)\leq \cl^*(C_i^\perp)=C_i^\perp$ and $\cl^*(W^\perp)\not\leq C_i^\perp$.
This implies $\cl^*(V^\perp)\lneq \cl^*(W^\perp)$ and thus
\[
    \rho^*(V^\perp)=\rho^*(\cl^*(V^\perp))<\rho^*(\cl^*(W^\perp))=\rho^*(W^\perp).
\]
Since $\dim V^\perp+1=\dim W^\perp$, submodularity of~$\rho^*$ yields $\rho^*(W^\perp)\leq \rho^*(V^\perp)+1$.
All of this leads to $\rho^*(W^\perp)= \rho^*(V^\perp)+1$.
Now we compute
\begin{align*}
  \rho(V)&=\dim V+\rho^*(V^\perp)-\rho^*(E)=\dim V +\rho^*(W^\perp)-1-\rho^*(E)\\
             &=\dim W+\rho^*(W^\perp)-\rho^*(E)=\rho(W),
\end{align*}
as desired. All of this proves that $x\in\cyc(V)$ and thus $\sum_{i=1}^t C_i\leq \cyc(V)$.
Finally, $\cyc(V)=0$ if and only if $V$ contains no circuits, which means that~$V$ is independent.
\end{proof}

The cyclic-core operator is dual to the closure operator in the following sense.
This also appears in \cite[Lem.~2.23]{AlBy22}.

\begin{cor}\label{C-cyccl}
Let $V\in\cL(E)$. Then $\cyc(V)^\perp=\cl^*(V^\perp)$.
\end{cor}

\begin{proof}
Let $C_1,\ldots,C_t\in\cC(\cM)$ be the circuits contained in~$V$.
With the aid of \cref{L-dualClos} and \cref{T-CycCoreOpen} we compute
$\cl^*(V^\perp)= \bigcap_{i=1}^t C_i^\perp=\big(\sum_{i=1}^t C_i\big)^\perp=\cyc(V)^\perp$.
\end{proof}

The previous results imply that  open spaces of~$\cM$ form a lattice $(\cO(\cM),<,\wedge,\vee)$ with $V\wedge W=\cyc(V\cap W)$ and
$V\vee W=V+W$. It is the dual of the lattice of flats of~$\cM^*$; for the latter see \cite[Thm.~3.10 and 3.13]{BCIJS23}.
We do not need this result in this paper.

The following is immediate with \cref{T-CycCoreOpen} or \cref{C-cyccl}; see also~\cite[Thm.~2.10]{AlBy22}.

\begin{cor}\label{P-CycProp}
Let $V,W\in\cL(E)$. Then
\begin{alphalist}
\item $V\leq W\Longrightarrow \cyc(V)\leq\cyc(W)$.
\item $\cyc(\cyc(V))=\cyc(V)$.
\end{alphalist}
\end{cor}

In the context of the cyclic core, it is natural to cast the following definition.

\begin{defi}
A subspace $V\in\cL(E)$ is \emph{cyclic} if $\cyc(V)=V$.
\end{defi}

\cref{T-CycCoreOpen}  and \cref{D-CycCore} show that any $V\in\cL(E)$ satisfies
\begin{equation}\label{e-cyclicopen}
  V\text{ is open }\Longleftrightarrow V\text{ is cyclic }\Longleftrightarrow
  \rho(W)=\rho(V)\text{ for all $W\in\Hyp(V)$.}
\end{equation}
We will use ``open'' and ``cyclic'' interchangeably and use the notation $\cO(\cM)$ for the collection of cyclic spaces of~$\cM$.

Dualizing the identity $\rho(V)=\rho(\cl(V))$ provides us with part~(a) of the next result;  see also \cite[Lem.~2.16 and~2.17]{AlBy22}.

\begin{prop}\label{P-VcycVFormula}
Let $V\in\cL(E)$.
\begin{alphalist}
\item $\dim V-\rho(V)=\dim\cyc(V)-\rho(\cyc(V))$.
\item Let $V=\cyc(V)\oplus W$. Then $\rho(V)=\rho(\cyc(V))+\dim W$ and $\dim W=\rho(W)$, i.e.,~$W$ is independent.
\end{alphalist}
\end{prop}

\begin{proof}
(a) follows from
\begin{align*}
   \dim V^\perp+\rho(V)-\rho(E)&=\rho^*(V^\perp)=\rho^*(\cl^*(V^\perp))=\rho^*(\cyc(V)^\perp)\\
      &=\dim\cyc(V)^\perp+\rho(\cyc(V))-\rho(E).
\end{align*}
(b) The first part is a consequence of~(a) because $\dim W=\dim V-\dim\cyc(V)$.
Now submodularity implies $\rho(\cyc(V))+\dim W=\rho(V)\leq \rho(\cyc(V))+\rho(W)$, and thus $\rho(W)=\dim W$.
\end{proof}

\section{The Lattice of Cyclic Flats}\label{S-CycFlats}
We turn to the notion of cyclic flats, which are simply flats that are also cyclic spaces.
Thanks to the closure operator and cyclic-core operator, the collection of cyclic flats turns into a lattice.
The main result of this section states that the collection of cyclic flats along with their rank values uniquely determine the \qM{}.
This section is the relatively standard $q$-analogue of the theory of cyclic flats for matroids as can be found in \cite{BoDM08} or
\cite[Sec.~2.4]{FHHW18}.
The first two results appear also, with slightly different proofs, in \cite[Sec.~2.2]{AlBy22}.
Thereafter we take a different route than~\cite{AlBy22} by focusing on the rank function.

We continue with the setting from \cref{Nota}.
Our first result states that the cyclic-core operator preserves flatness and the closure operator preserves cyclicity.
The corresponding fact for classical matroids is only mentioned in passing at \cite[Sec.~2.4]{FHHW18}.

\begin{lemma}\label{L-FlatCyclic}
Recall the collections $\cF(\cM)$ of flats and $\cO(\cM)$ of open (i.e., cyclic) spaces of~$\cM$.
\begin{alphalist}
\item $F\in\cF(\cM)\Longrightarrow\cyc(F)\in\cF(\cM)$.
\item $F\in\cO(\cM)\Longrightarrow\cl(F)\in\cO(\cM)$.
\end{alphalist}
As a consequence, every $V\in\cL(E)$ satisfies $\cl(\cyc(V))\leq \cyc(\cl(V))$ and both spaces  are elements of the intersection $\cF(\cM)\cap\cO(\cM)$.
\end{lemma}

The spaces $\cl(\cyc(V))$ and $\cyc(\cl(V))$ are in general not identical. See \cref{E-UniMatr} below.

\begin{proof}
(a) Let $x\in E\setminus\cyc(F)$. We have to show that $\rho(\cyc(F)+\subspace{x})>\rho(\cyc(F))$.
\\
i) If $x\not\in F$, then $\rho(F+\subspace{x})=\rho(F)+1$, since $F$ is a flat, and the desired inequality follows from \cref{P-RhoProp}(b).
\\
ii) Let $x\in F\setminus\cyc(F)$. The definition of $\cyc(F)$ implies the existence of a space $\hat{W}\in\Hyp(F)$ such that
$x\not\in\hat{W}$ and $\rho(\hat{W}+\subspace{x})=\rho(\hat{W})+1$. Thus, $\rho(F)=\rho(\hat{W})+1$.
Assume by contradiction that $\rho(\cyc(F)+\subspace{x})=\rho(\cyc(F))$.
Then \cref{P-RhoProp}(b) implies $\rho(W+\subspace{x})=\rho(W)$ for all subspaces~$W$ containing $\cyc(F)$.
Hence we conclude $\cyc(F)\not\leq \hat{W}$.
Thus there exists $y\in\cyc(F)\setminus\hat{W}$. Now $\hat{W}+\subspace{y}=F$, and
$\rho(\hat{W})=\rho(F)$ since $y\in\cyc(F)$. This contradicts the above, and thus $\rho(\cyc(F)+\subspace{x})>\rho(\cyc(F))$.
\\
(b) Let $\cyc^*$ be the cyclic-core operator of the dual \qM~$\cM^*$. With the aid of \cref{T-Basics}(a), \cref{C-cyccl} and Part (a) we conclude
\[
   F\in\cO(\cM)\Rightarrow F^\perp\in\cF(\cM^*)\Rightarrow\cyc^*(F^\perp)\in\cF(\cM^*)\Rightarrow
  \cl(F)^\perp\in\cF(\cM^*)\Rightarrow\cl(F)\in\cO(\cM).
\]
As for the consequence note that $\cl(V)$ is a flat and thus (a) together with \cref{P-CycProp} implies that $\cyc(\cl(V))$ is a flat containing $\cyc(V)$.
Now the stated containment follows from \cref{T-Basics}(e). The rest is clear.
\end{proof}

Using the cyclic-core operator and the closure operator we obtain a lattice consisting of the cyclic flats.
It also appears in \cite[Prop.~2.24]{AlBy22} and is the $q$-analogue of \cite[Prop.~3]{FHGHW21} (which refers to \cite{BoDM08}).

\begin{cor}\label{C-CycFlats}
Let $\cZ(\cM)=\cF(\cM)\cap\cO(\cM)$, that is, $\cZ(\cM)$ is the collection of cyclic flats of~$\cM$, or, alternatively, of open and closed spaces.
Then $(\cZ(\cM),\leq,\wedge,\vee)$ is a lattice, where the meet and join are defined as
\[
   Z_1\wedge Z_2=\cyc(Z_1\cap Z_2)\ \text{ and }\ Z_1\vee Z_2=\cl(Z_1+Z_2)\ \text{ for all }Z_1,Z_2\in\cZ(\cM).
\]
The rank values of the meet and join are given by
\[
  \rho(Z_1\wedge Z_2)= \rho(Z_1\cap Z_2)-\dim\big((Z_1\cap Z_2)/Z_1\wedge Z_2\big) \ \text{ and }\ \rho(Z_1\vee Z_2)=\rho(Z_1+Z_2).
\]
As a consequence
\[
   \rho(Z_1)+\rho(Z_2)\geq \rho(Z_1\vee Z_2)+\rho(Z_1\wedge Z_2)+\dim\big((Z_1\cap Z_2)/Z_1\wedge Z_2\big).
\]
\end{cor}

\begin{proof}
By \cref{T-Basics}(d) $Z_1\cap Z_2$ is a flat and $Z_1+Z_2$ is an open space. Hence $Z_1\wedge Z_2$ and $Z_1\vee Z_2$ are in $\cZ(\cM)$ thanks to \cref{L-FlatCyclic}.
Next, if $V\in\cZ(\cM)$ satisfies $V\leq Z_i$ for $i=1,2$, then $V$ is an open space in $Z_1\cap Z_2$ and thus $V\leq\cyc(Z_1\cap Z_2)$ thanks to
\cref{T-CycCoreOpen}.
Similarly, if $W\in\cZ(\cM)$ satisfies $Z_i\leq W$ for $i=1,2$, then~$W$ is a closed space containing $Z_1+Z_2$ and thus $\cl(Z_1+Z_2)\leq W$ by
\cref{T-Basics}(e). Thus  $(\cZ(\cM),\leq,\wedge,\vee)$ is a lattice.
The rank value of $Z_1\vee Z_2$ follows from \eqref{e-clrho}, while the rank value of $Z_1\wedge Z_2$ is a consequence of \cref{P-VcycVFormula}(a).
The last inequality now follows from submodularity of~$\rho$ applied to $Z_1\cap Z_2$ and $Z_1+Z_2$.
\end{proof}

Note that $\cZ=\cZ(\cM)$ is not empty.
The least and greatest elements of  the lattice $\cZ$ are given by $0_{\cZ}=\cl(0)$ and $1_{\cZ}=\cyc(E)$ and their rank values are
\[
    \rho(0_{\cZ})=0\ \text{ and }\ \rho(1_{\cZ})=\rho(E)-\dim(E/\cyc(E)),
\]
where the second part follows from \cref{P-VcycVFormula}.
In \cref{E-CycFlatsStart} below we will see that the lattice~$\cZ$ is in general not semi-modular, and thus not graded.
But even in the case where $\cZ$ is graded, its height function does not agree, in general, with the rank function~$\rho$.

\begin{exa}[see also \mbox{\cite[Prop.~2.30]{AlBy22}}]\label{E-UniMatr}
Let $0< k< n=\dim E$ and consider the uniform \qM{} $\cU_k(E)$; see \cref{D-Unif}.
The only flats other than $E$  are the spaces of dimension at most $k-1$, and the only nonzero cyclic spaces are the spaces of
dimension at least $k+1$; see \eqref{e-cyclicopen}.
Thus $\cZ:=\cZ(\cU_k(E))=\{0,E\}$.
Since this is true for every~$k$, this shows that the collection $\cZ(\cM)$ of a \qM{}~$\cM$ does not uniquely determine~$\cM$.
Furthermore, since $\rho(E)=k$, we also see that the height function of the lattice $\cZ$
does not agree with the rank function of~$\cU_{k}(E)$ (unless $k=1$).
Finally, for any $k$-dimensional subspace~$V$ we have $\cl(\cyc(V))=\cl(0)=0$, and
$\cyc(\cl(V))=\cyc(E)=E$.
This shows that $\cl(\cyc(V))$ and $\cyc(\cl(V))$ do not agree.
For the trivial and the free \qM{} we have $\cZ(\cU_0(E))=\{E\}$ and $\cZ(\cU_n(E))=\{0\}$.
\end{exa}

Below we will show that we can reconstruct the entire \qM{} $\cM$ by way of the cyclic flats along with their rank values.
The following lemma will be needed. It is the $q$-analogue of \cite[Lem.~5]{FHGHW21} and also appears in \cite[Lem.~2.23]{AlBy22},
where it is proven with the aid of a characterization of cyclic spaces as inclusion-minimal spaces with respect to the nullity function.

\begin{lemma}\label{L-VclcycV}
Let $V\in\cL(E)$. Then $V\cap\cl(\cyc(V))=\cyc(V)$.
\end{lemma}

\begin{proof}
``$\supseteq$'' is clear. For the other containment let $x\in V\cap \cl(\cyc(V))$.  Then $\subspace{x}+\cyc(V)\leq V$ and
$\rho(\subspace{x}+\cyc(V))=\rho(\cyc(V))$ because $x$ is in the closure of $\cyc(V)$.
Furthermore,
$\cyc(V)=\cyc(\cyc(V))\leq\cyc(\subspace{x}+\cyc(V))\leq\cyc(V)$, and thus $\cyc(V)=\cyc(\subspace{x}+\cyc(V))$.
With the aid of \cref{P-VcycVFormula}(a) we compute
\begin{align*}
   \dim\cyc(V)-\rho(\cyc(V))&\leq \dim(\subspace{x}+\cyc(V))-\rho(\cyc(V))\\
   &=\dim(\subspace{x}+\cyc(V))-\rho(\subspace{x}+\cyc(V))\\
   &=\dim(\cyc(\subspace{x}+\cyc(V)))-\rho(\cyc(\subspace{x}+\cyc(V)))\\
   &=\dim\cyc(V)-\rho(\cyc(V)).
\end{align*}
Hence we have equality in the first step, and this means $x\in\cyc(V)$.
\end{proof}

Now we arrive at the following $q$-analogue of \cite[Lem.~3.1(i)]{BoDM08} characterizing independent spaces.
This characterization will be crucial because it
will allow us to derive the entire rank function of the \qM{} from the cyclic flats and their rank values.
This is different from the approach taken in \cite{BoDM08} and in \cite{AlBy22} for \qM{}s, where
the entire ($q$-)matroid is reconstructed through the lattice of flats.

\begin{theo}\label{T-ZDeterminesM}
Consider the collection $\cZ(\cM)$ of cyclic flats of~$\cM$. Let $V\in\cL(E)$. Then
\[
    V\text{ is independent}\Longleftrightarrow \dim(V\cap Z)\leq \rho(Z)\ \text{ for all }\ Z\in\cZ(\cM).
\]
Thus the cyclic flats together with their rank values fully determine the collection of independent spaces and thus the entire \qM{} $\cM$.
\end{theo}

\begin{proof}
``$\Rightarrow$'' If~$V$ is independent, then so is $V\cap Z$ and thus $\dim(V\cap Z)=\rho(V\cap Z)\leq\rho(Z)$ for any subspace~$Z$.
\\
``$\Leftarrow$'' Let~$V$ be dependent and $Z=\cl(\cyc(V))$, which is in $\cZ(\cM)$.
Then $V\cap Z=\cyc(V)$ thanks to \cref{L-VclcycV} and $\rho(Z)=\rho(\cyc(V))$ by \eqref{e-clrho}.
Now \cref{P-VcycVFormula}(a) implies  $\dim(V\cap Z)=\dim\cyc(V)=\rho(Z)+\dim V-\rho(V)>\rho(Z)$, where the last step follows from
the dependence of~$V$.
This establishes the equivalence.
The last statement follows from the well-known fact that the independent spaces fully determine the \qM{}; see \eqref{e-rhoVI}.
\end{proof}

The previous characterization of independent spaces lets us determine the rank function of the entire \qM{}  from the cyclic
flats together with their rank values; see also \cite[Prop.~3]{FHHW18} for classical matroids.

\begin{cor}\label{C-RhoVFormula}
Let $\cM=(E,\rho)$ be a \qM{} and $\cZ=\cZ(\cM)$ be its collection of cyclic flats.
Then
\[
    \rho(V)=\min_{Z\in\cZ}\Big(\rho(Z)+\dim(V+Z)/Z\Big)\ \text{ for all }\ V\in\cL(E).
\]
\end{cor}

\begin{proof}
\cref{T-ZDeterminesM} tells us that a space~$I$ is independent if and only if $\rho(Z)\geq \dim(I\cap Z)$ for all $Z\in\cZ$.
With the aid of the dimension formula for subspaces we may rewrite the inequality as $\dim I\leq \rho(Z)+\dim(I+Z)/Z$.
Now let $V\in\cL(E)$.
Using \eqref{e-rhoVI} we obtain
\begin{align*}
    \rho(V)&=\max\{\dim I\mid I\leq V,\ I\text{ independent}\}\\
       &=\max\{\dim I\mid I\leq V,\,\dim I\leq \rho(Z)+\dim(I+Z)/Z\text{ for all }Z\in\cZ\}\\
         &\leq \max\{\dim I\mid I\leq V,\,\dim I\leq \rho(Z)+\dim(V+Z)/Z\text{ for all }Z\in\cZ\}\\
         &\leq \min_{Z\in\cZ}\Big(\rho(Z)+\dim(V+Z)/Z\Big).
\end{align*}
For the converse consider $\hat{Z}=\cl(\cyc(V))$, which is in~$\cZ$.
With the aid of \eqref{e-clrho}, \cref{P-VcycVFormula}(a) and \cref{L-VclcycV} we arrive at
\begin{align*}
    \min_{Z\in\cZ}\Big(\rho(Z)+\dim(V+Z)/Z\Big)&\leq\rho(\hat{Z})+\dim(V+\hat{Z})/\hat{Z}\\
     &=\rho(\hat{Z})+\dim V-\dim(V\cap\cl(\cyc(V))\\
      &=\rho(\cyc(V))+\dim V-\dim\cyc(V)\\
      &=\rho(V).\qedhere
\end{align*}
\end{proof}

The above result does not imply that the lattice structure of $(\cZ(\cM),\leq,\wedge,\vee)$
together with the rank values of the cyclic flats are sufficient to determine the \qM{} up to equivalence.
This will be illustrated in \cref{E-Dim4}.

The collection of cyclic flats is often astoundingly small.
The following example is inspired by \cite[Sec.~3]{AlBy22}.

\begin{exa}\label{E-CycFlatsStart}
Let $\F=\F_2$ and in $\F^8$ consider the collection $\cZ=\{Z_0,\ldots,Z_4\}$, where
\[
  Z_0=0,\ Z_1=\subspace{e_1,e_2},\ Z_2=\subspace{e_1,e_2,e_3,e_4},\ Z_3=\subspace{e_5,e_6,e_7,e_8},\ Z_4=\F^8.
\]
Set $\hat{\rho}(Z_i)=i$ for $i=0,\ldots,4$.
Using a computer algebra system one verifies that the map
\[
   \rho:\cL(\F^8)\longrightarrow\N_0,\quad \rho(V)=\min_{Z\in\cZ}\Big(\hat{\rho}(Z)+\dim(V+Z)/Z\Big),
\]
satisfies (R1)--(R3) from \cref{D-qMatroid}  (see also \cref{C-RhoVFormula}).
Moreover, $\rho(Z)=\hat{\rho}(Z)$ for all $Z\in\cZ$ and the \qM{} $\cM=(\F^8,\rho)$ satisfies $\cZ(\cM)=\cZ$.
This also follows from \cite[Prop.~3.7 and Thm.~3.11]{AlBy22}.
We have the following cardinalities of flats, cyclic spaces etc.
\[
   \begin{array}{|c|c|c|c|c|c|c|}
   \hline
        \text{flats}&\text{cyclic spaces}&\text{cyclic flats}&\text{ind.\ spaces}&\text{dep.\ spaces}&\text{circuits}&\text{bases}\\ \hline\hline
        \rule[-.15cm]{0cm}{0.63cm}   99597 &105097  &5 &307905& 109294 &94079&199775\\ \hline
   \end{array}
\]
Finally, notice that the lattice $\cZ$ has the form
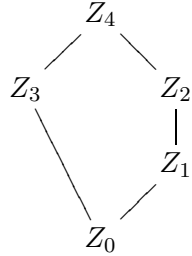
\[
\begin{array}{l}
   \begin{xy}
   (10,30)*+{Z_4}="c4";
   (20,10)*+{Z_1}="c1";%
   (20,20)*+{Z_2}="c2";%
    (0,20)*+{Z_3}="c3";%
   (10,0)*+{Z_0}="c0";%
    {\ar@{-}"c0";"c1"};
    {\ar@{-}"c0";"c3"};
    {\ar@{-}"c1";"c2"};
    {\ar@{-}"c2";"c4"};
    {\ar@{-}"c3";"c4"};
   \end{xy}
\end{array}
\]
and this is clearly not semi-modular (see \cite[Prop.~3.3.2]{Sta97}).
Interestingly enough, $\cM$ is representable.
A carefully crafted random search leads to the matrix
\[
    G=\begin{pmatrix}
    1 & \omega^{26772} & 0 & \omega^{43180} & 0 & 0 & \omega^{46265} & \omega^{31452} \\
0 & 0 & 1 & \omega^{3844} & 0 & 0 & \omega^{8371} & \omega^{59093} \\
0 & 0 & 0 & 0 & 1 & 0 &\omega^{45712} &\omega^{50716} \\
0 & 0 & 0 & 0 & 0 & 1 & \omega^{12688} & \omega^{10916}
    \end{pmatrix}\in\F_{2^{16}}^{4\times8}
\]
where $\omega$ is a primitive element satisfying $\omega^{16} + \omega^5 + \omega^3 + \omega^2 + 1=0$.
It is not clear to us whether~$\cM$ is representable over a smaller field.
\end{exa}

We close this section with the most extreme case.
Every \qM{} has at least one cyclic flat, namely $\cl(0)$.
Let us consider the case where this is the sole cyclic flat.
Such \qM{}s do indeed exist, for instance, the trivial and the free \qM{} on~$E$ (see \cref{E-UniMatr}) or the \qM{} $\cM=\cM_G$, where
\[
    G=\begin{pmatrix}1&0&0\\0&1&0\end{pmatrix}\in\F_2^{2\times 3},
\]
which has the single cyclic flat $\subspace{e_3}$ (where $e_i$ denotes the $i$th standard basis vector).

\begin{prop}\label{P-ExtremeExa}
Let $\cM=(E,\rho)$ be a \qM{} with a single cyclic flat, say $\cZ(\cM)=\{\hat{Z}\}$. Thus $\hat{Z}=\cl(0)=\cyc(E)$ and $\rho(\hat{Z})=0$.
Then for any $V\in\cL(E)$
\[
   V\in\cF(\cM)\Longleftrightarrow \hat{Z}\leq V\quad \text{ and }\quad V\in\cO(\cM)\Longleftrightarrow V\leq\hat{Z}.
\]
In particular, $\cl(0)=E\Longleftrightarrow\cM=\cU_0(E)\ \text{ and }\ \cyc(E)=0\Longleftrightarrow \cM=\cU_{\dim E}(E)$.
\end{prop}

\begin{proof}
Note first that for any $V\in\cL(E)$ we have
\begin{equation}\label{e-SingleZ}
   \cl(\cyc(V))=\cyc(\cl(V))=\hat{Z}.
\end{equation}
Thus if~$V$ is a flat, then $\hat{Z}=\cyc(V)\leq V$. Similarly, if~$V$ is cyclic, then $V\leq\cl(V)=\hat{Z}$.
It remains to consider the opposite implications.
\\
1) If $V\leq\hat{Z}$, then $\rho(V)=0$ and~$V$ is clearly cyclic.
\\
2) Let now $\hat{Z}\leq V$. Then \eqref{e-SingleZ} along with \cref{L-VclcycV} implies $\cyc(V)=\hat{Z}$. Writing $V=V_1\oplus\hat{Z}$, we obtain from
\cref{P-VcycVFormula}(b)
\[
  \rho(V)=\rho(\hat{Z})+\dim V_1=\dim V_1.
\]
Let now $x\in E\setminus V$. Then $\hat{Z}\leq V+\subspace{x}$ and we may write $V+\subspace{x}=(V_1+\subspace{x_1})\oplus\hat{Z}$ for some $x_1\not\in V_1$.
But then the same reasoning as for~$V$ provides us with $\rho(V+\subspace{x})=\dim (V_1+\subspace{x_1})$, and hence $\rho(V+\subspace{x})=\dim V_1+1=\rho(V)+1$.
Thus $V$ is a flat.
\\
The last two equivalences are clear.
\end{proof}

Later in \cref{C-SingleCF} we will see that a \qM{} with a single cyclic flat is the direct sum of a trivial and a free \qM.

\section{The Direct Sum of $q$-Matroids}\label{S-DirSum}
In this section we turn to the direct sum of \qM{}s.
The first definition of it has been given in~\cite{CeJu21}.
We will present a different, more concise, definition, which results in the same construction.
It will enable us to study its properties in more detail.
In particular, we will derive various ways to compute the rank function, the most efficient one being based on the cyclic flats of the components.
Moreover, we will show that the dual of a direct sum is the direct sum of the dual \qM{}s.

We start with the union of two \qM{}s, which is the $q$-analogue of the matroid union; see e.g., \cite[Thm.~11.3.1]{Ox11}.

\begin{theo}[\mbox{\cite[Thm.~28]{CeJu21}}]\label{T-Union}
Let $\cM_i=(E,\rho_i),\,i=1,2,$ be \qM{}s on the same ground space~$E$. For $V\in\cL(E)$ define
\begin{equation}\label{e-rhoUnion}
    \rho(V)=\dim V+ \min\{\rho_1(X)+\rho_2(X)-\dim X\mid X\leq V\}
\end{equation}
Then $\cM=(E,\rho)$ is a $q$-matroid, called the \emph{union} of~$\cM_1$ and $\cM_2$, and denoted by $\cM=\cM_1\vee\cM_2$.
\end{theo}

\begin{proof}
The fact that $\rho$ is a rank function is in \cite[Thm.~29]{CeJu21} and also follows
from \cite[Thm.~3.6 and its proof]{GLJ22Ind}.
The main argument in these proofs is that the map $\rho_1+\rho_2:\cL(E)\longrightarrow\N_0$ satisfies (R2) and (R3) from \cref{D-qMatroid}.
Hence it gives rise to a $q$-polymatroid.
The minimization in \eqref{e-rhoUnion} turns this map into a rank function of a \qM{}.
It mimics the analogous procedure that associates a matroid to a given polymatroid; see for instance \cite[Prop.~11.1.7]{Ox11}.
\end{proof}

In order to define the direct sum of two \qM{}s $\cM_i=(E_i,\rho_i)$, we need, unsurprisingly, the direct sum $E=E_1\oplus E_2$ of $\F$-vector spaces~$E_1$ and~$E_2$.
For any such direct sum we denote by
\[
      \pi_i:E\longrightarrow E_i\ \text{  and }\ \iota_i:E_i\longrightarrow E
\]
the corresponding projection and embedding.
We will identify a subspace $V_i\in\cL(E_i)$ with its image $\iota_i(V_i)$.

The construction of the direct sum consists of two steps:
\\[.5ex]
(1) Add the ground space $E_1$ of the \qM{} $\cM_1=(E_1,\rho_1)$  as a space of rank 0 to the \qM{} $\cM_2=(E_2,\rho_2)$ and vice versa.
This results in two \qM{}s $\cM'_1,\,\cM'_2$ on the common ground space $E_1\oplus E_2$ and with rank functions~$\rho'_i$.
\\[.5ex]
(2) Take the union of~$\cM'_1$ and $\cM'_2$.

\begin{theo}[\mbox{\cite[Sec.~7]{CeJu21}}]\label{T-DirSum}
Let $\cM_i=(E_i,\rho_i),\,i=1,2,$ be \qM{}s and set $E=E_1\oplus E_2$.
Define $\rho'_i:\cL(E)\longrightarrow \N_0,\ V\longmapsto \rho_i(\pi_i(V))$ for $i=1,2$.
Then $\cM'_i=(E,\rho'_i)$ is a \qM{} for $i=1,2$.
Set $\cM=\cM'_1\vee\cM'_2$, that is, $\cM=(E,\rho)$, where
\begin{equation}\label{e-rho}
  \rho(V)= \dim V+\min_{X\leq V}\big(\rho'_1(X)+\rho'_2(X)-\dim X\big)
  \ \text{ for } V\in\cL(E).
\end{equation}
Then $\cM$ is a $q$-matroid, called the \emph{direct sum of $\cM_1$ and~$\cM_2$} and denoted by $\cM_1\oplus\cM_2$.
\end{theo}

\begin{proof}
A short proof showing that~$\cM'_i$ is a \qM{} can be found in \cite[Thm.~5.2]{GLJ21C} .
The rest follows from \cref{T-Union}.
\end{proof}

We clearly have $\rho'_i(E_j)=0$ for $i\neq j$.
Even more, one can easily verify that $\cM'_i\approx\cM_i\oplus\cU_0(E_j)$, where $\cU_0(E_j)$ is the trivial $q$-matroid on the ground
space $E_j$. Thus~$\cM'_i$ is a special instance of the direct sum (called ``adding a loop space'' in \cite{CeJu21}).
Some additional basic and to-be-expected properties of the direct sum will be presented below in \cref{T-DirSumProp} after deriving
more convenient expressions for the rank function.

Before doing so, we mention the following result, which  implies that $\cM_1\oplus\cM_2$ is the $q$-matroid on $E=E_1\oplus E_2$ with the
``least amount of conditions'' among all $q$-matroids on~$E$ whose restriction to~$E_i$ is isomorphic to~$\cM_i$ for $i=1,2$.

\begin{theo}[\mbox{\cite[Thm.~5.5]{GLJ21C}}]\label{T-DirSumCoPr}
Let $\cM_i=(E_i,\rho_i),\,i=1,2,$ be \qM{}s and $\cM=\cM_1\oplus\cM_2=(E,\rho)$ be the direct sum.
Then $(\cM,\iota_1,\iota_2)$ is a coproduct of~$\cM_1$ and~$\cM_2$ in the category of $q$-matroids with linear weak maps as morphisms.
In other words, if $\cN=(E,\tau)$ is any $q$-matroid satisfying $\tau(V_i)\leq \rho_i(V_i)$ for all $V\in\cL(E_i)$ and $i=1,2$, then
$\rho(V)\geq\tau(V)$ for all $V\in\cL(E)$.
As a consequence, the independent spaces satisfy $\cI(\cN)\subseteq\cI(\cM)$.
\end{theo}

The definition of the rank function of $\cM_1\oplus\cM_2$ in \eqref{e-rho} becomes quickly very cumbersome as it requires computing
the minimum over all subspaces of~$V$.
We now derive various more efficient ways of determining the rank values.
The most convenient one is given in \cref{C-RhoDSZ} below, which only requires the cyclic flats of $\cM_1$ and $\cM_2$.
We start with the following simple improvements.

\begin{prop}\label{P-XList}
Consider the situation of \cref{T-DirSum}.
Define the sets
\begin{align}
    \cX&=\{X\in\cL(E)\mid \rho_1'(X)+\rho_2'(X)<\dim X\},  \label{e-XSet}\\
    \cT&=\{X_1\oplus X_2\mid X_i\in\cL(E_i)\}, \nonumber\\
    \cT(V)&=\{X_1\oplus X_2\mid X_i\leq\pi_i(V)\}\ \text{ for }\ V\in\cL(E).\nonumber
\end{align}
Then for any $V\in\cL(E)$
\begin{align}
   \rho(V)&=\dim V+\min_{X\in\{0\}\cup(\cX\cap\cL(V))}\big(\rho'_1(X)+\rho'_2(X)-\dim X\big)   \label{e-rhoVX}\\[.6ex]
      &=\dim V+\min_{X\in\cT}\big(\rho_1'(X)+\rho_2'(X)-\dim(X\cap V)\big) \label{e-rhoVT}\\[.6ex]
      &=\dim V+\min_{X\in\cT(V)}\big(\rho_1'(X)+\rho_2'(X)-\dim(X\cap V)\big).\label{e-rhoVTV}
\end{align}
As a consequence, $V$ is independent in $\cM_1\oplus\cM_2$ iff $\cL(V)\cap\cX=\emptyset$.
\end{prop}

\begin{proof}
1) Using $X=0$ in \eqref{e-rho} we obtain $\rho(V)\leq \dim V$ (as it has to be).
Hence we only have to test the subspaces~$X$ of~$V$ that may lead to $\rho(V)<\dim V$. This is exactly the collection~$\cX$, and thus
\eqref{e-rhoVX} is established. This identity also implies the statement about the independent spaces.
\\
2)
Let $M_1:=\min_{X\leq V}\big(\rho'_1(X)+\rho'_2(X)-\dim X\big)$ and $M_2:=\min_{X\leq E}\big(\rho'_1(X)+\rho'_2(X)-\dim(X\cap V)\big)$.
Clearly  $M_2\leq M_1$.
For the converse inequality let $\hat{X}\leq E$ be such that
$M_2=\rho'_1(\hat{X})+\rho'_2(\hat{X})-\dim(\hat{X}\cap V)$.
Set $\tilde{X}:=\hat{X}\cap V$.
Using monotonicity of the rank function we obtain
$M_1\leq\rho'_1(\tilde{X})+\rho'_2(\tilde{X})-\dim\tilde{X}\leq \rho'_1(\hat{X})+\rho'_2(\hat{X})-\dim(\hat{X}\cap V)=M_2$.
Thus $M_1=M_2$ and
\begin{equation}\label{e-M1M2}
  \rho(V)=\dim V+\min_{X\leq E}\big(\rho_1'(X)+\rho_2'(X)-\dim(X\cap V)\big).
\end{equation}
Actually, this expression for the rank function of the direct sum appears in \cite[Thm.~25]{CeJu21}; see also \cite[Rem.~26]{CeJu21}.
Now we are ready to prove \eqref{e-rhoVT}.
Recall that $\rho'_i(X)=\rho_i(\pi_i(X))$.  Hence $\rho_1'(X)+\rho_2'(X)$ only depends on the projections $\pi_i(X)$.
Since $X\leq\pi_1(X)\oplus\pi_2(X)=:X'$ for any subspace~$X$, we have
$\rho_1'(X)+\rho_2'(X)-\dim(X\cap V)\geq \rho_1'(X')+\rho_2'(X')-\dim(X'\cap V)$.
This shows that it suffices to take the minimum in \eqref{e-M1M2} over subspaces~$X$ satisfying $X=\pi_1(X)\oplus\pi_2(X)$.
But this is exactly the collection~$\cT$, and hence \eqref{e-rhoVT} is established.
\\
3) Let $X=X_1\oplus X_2\in\cT$. Then $X\cap V\subseteq Y_1\oplus Y_2$, where $Y_i=X_i\cap \pi_i(V)$.
Set $Y:=Y_1\oplus Y_2$. Then $Y\in\cT(V)$ and $X\cap V\subseteq Y\cap V$. Thus
$\rho_1(Y_1)+\rho_2(Y_2)-\dim(Y\cap V)\leq \rho_1(X_1)+\rho_2(X_2)-\dim (X\cap V)$.
This shows that the minimum in \eqref{e-rhoVT} is attained by a subspace in $\cT(V)$ and \eqref{e-rhoVTV} is proven.
\end{proof}

The above allows an immediate characterization of the circuits of the direct sum.
The result can already be found in \cite[Thm.~5.2]{GLJ21C}.

\begin{cor}\label{C-CircuitsDS}
Consider the direct sum $\cM_1\oplus\cM_2$ as in \cref{T-DirSum} and the set~$\cX$ in \eqref{e-XSet}.
Then the circuits of the direct sum are given by
\[
   \cC(\cM_1\oplus\cM_2)=\{X\in\cX\mid X \text{ is inclusion-minimal in }\cX\}.
\]
\end{cor}

As to be expected, the direct sum behaves well with respect to restriction to the initial ground spaces~$E_i$ and the according contraction.
Recall Definitions~\ref{D-Equiv} and~\ref{D-RestrContr}.

\begin{theo}[\mbox{\cite[Thm.~47, Cor.~48]{CeJu21}}]\label{T-DirSumProp}
Let $\cM=\cM_1\oplus\cM_2$ be as in \cref{T-DirSum}.
\begin{alphalist}
\item For all $V_i\in\cL(E_i)$ we have $\rho_i(V_i)=\rho'_i(V_i)$ and $\rho'_j(V_i)=0$ for $j\neq i$.
         Furthermore, $\rho(V_1\oplus V_2)=\rho_1(V_1)+\rho_2(V_2)$.
         Hence  $\cM_i\approx \cM|_{E_i}\approx(\cM'_i)|_{E_i}$, and the isomorphism is provided by~$\iota_i$.
         Moreover, for $i\neq j$ the $q$-matroid $(\cM'_i)|_{E_j}$ is the zero $q$-matroid.
\item $\rho(\cM)=\rho_1(\cM_1)+\rho_2(\cM_2)$.
\item $\cM/E_i\approx\cM_j$ for $i\neq j$.
\end{alphalist}
\end{theo}

\begin{proof}
(a) The first statement follows from $\rho_i'(V_i)=\rho_i(\pi_i(V_i))=\rho_i(V_i)$.
Let now $V=V_1\oplus V_2$.
By \eqref{e-rhoVTV} there exist $X=X_1\oplus X_2$ for some $X_i\leq V_i$ such that
$\rho(V)=\dim V+\rho_1'(X_1)+\rho_2'(X_2)-\dim X$.
Write $V_i=X_i\oplus Z_i$ for some $Z_i\leq V_i$.
Using the properties of the rank functions $\rho_i'$ we compute
\begin{align*}
   \rho(V)&=\rho_1'(X_1)+\rho_2'(X_2)+\dim Z_1+\dim Z_2\geq \rho_1'(X_1)+\rho_2'(X_2)+\rho_1'(Z_1)+\rho_2'(Z_2)\\
       &\geq \rho_1'(X_1\oplus Z_1)+\rho_2'(X_2\oplus Z_2)=\rho_1'(V_1)+\rho_2'(V_2)\\
       &=\rho_1(V_1)+\rho_2(V_2)\geq\rho(V),
\end{align*}
where the last step follows from \eqref{e-rhoVTV} with $X=V$.
This establishes the desired identity.
The isomorphisms follow from $\rho(V_i)=\rho_i(V_i)=\rho'_i(V_i)$, and the rest is clear.
\\
(b) is a consequence of~(a) because $\rho(\cM)=\rho(E)=\rho_1(E_1)+\rho_2(E_2)$.
\\
(c) Without loss of generality let $i=1$.
Denote the rank function of $\cM/E_1$ by $\hat{\rho}_1$.
Let $V\in\cL(E_2)$. Then
$\hat{\rho}_1(V/E_1)=\rho(E_1\oplus V)-\rho(E_1)=\rho_1(E_1)+\rho_2(V)-\rho_1(E_1)=\rho_2(V)$,
where the last identity follows from~(a).
Hence $\cM_2$ and$\cM/E_1$ are equivalent via the isomorphism $\hat{\pi}:E_2\longrightarrow E/E_1$ induced by the
canonical projection from~$E$ onto $E/E_1$.
\end{proof}

Now we obtain a very efficient way of computing the rank function of the direct sum which only requires the collections
of cyclic flats of the two summands. The following notation we be convenient for the rest of this paper.

\begin{nota}\label{N-DSCollections}
Let $E=E_1\oplus E_2$ and $\cY_i$ be collections of subspaces of~$E_i$ for $i=1,2$. We define
$\cY_1\oplus\cY_2=\{Y_1\oplus Y_2\mid Y_i\in\cY_i\}$.
\end{nota}

\begin{theo}\label{T-rhoVFlatsCyc}
Let $\cM_i=(E_i,\rho_i),\,i=1,2,$ be \qM{}s and
let $\cF_i=\cF(\cM_i)$ and $\cZ_i=\cZ(\cM_i)$, that is,~$\cF_i$ (resp.\ $\cZ_i$) are the collections of flats (resp.\ cyclic flats) of~$\cM_i$.
Let $\rho$ be the rank function of $\cM_1\oplus\cM_2$. Then for all $V\in\cL(E_1\oplus E_2)$ we have
\begin{align}
     \rho(V)&=\dim V+\min_{F_1\oplus F_2\in\cF_1\oplus\cF_2}\big(\rho_1(F_1)+\rho_2(F_2)-\dim((F_1\oplus F_2)\cap V)\big)\label{e-rhoVF}\\[.6ex]
       &=\dim V+\min_{Z_1\oplus Z_2\in\cZ_1\oplus\cZ_2}\big(\rho_1(Z_1)+\rho_2(Z_2)-\dim((Z_1\oplus Z_2)\cap V)\big).\label{e-rhoVZ}
\end{align}
\end{theo}

\begin{proof}
1) We use  \eqref{e-rhoVT}.
Let $X=X_1\oplus X_2\in\cT$.
Clearly, $\rho'_i(X)=\rho_i(X_i)$ for $i=1,2$.
Suppose $X_1\not\in\cF_1$. Then there exists $X_1'\in\cL(E_1)$ such that $X_1\lneq X_1'$ and
$\rho_1(X_1)=\rho_1(X_1')$. Therefore $X\lneq X_1'\oplus X_2=:X'$ and
$\rho_1(X_1')+\rho_2(X_2)-\dim(X'\cap V)\leq\rho_1(X_1)+\rho_2(X_2)-\dim(X\cap V)$.
The same argument applies to~$X_2$, and this shows that the minimum in \eqref{e-rhoVT} is attained by a space in $\cF_1\oplus\cF_2$.
This establishes \eqref{e-rhoVF}.
\\
2) Let $V\in\cL(E)$ and set
\begin{align*}
   M_1&=\min_{F_1\oplus F_2\in\cF_1\oplus\cF_2}\big(\rho_1(F_1)+\rho_2(F_2)-\dim((F_1\oplus F_2)\cap V)\big),\\
 M_2&=\min_{Z_1\oplus Z_2\in\cZ_1\oplus\cZ_2}\big(\rho_1(Z_1)+\rho_2(Z_2)-\dim((Z_1\oplus Z_2)\cap V)\big).
\end{align*}
Clearly $M_1\leq M_2$. For the converse let $\hat{F}:=\hat{F}_1\oplus \hat{F}_2\in\cF_1\oplus\cF_2$ be such that
$M_1=\rho_1(\hat{F}_1)+\rho_2(\hat{F}_2)-\dim(\hat{F}\cap V)$.
Set $Z_i=\cyc_i(\hat{F}_i)$, where $\cyc_i(\,\cdot\,)$ is the cyclic core in the \qM{}~$\cM_i$.
Then~$Z_i$ is in $\cZ_i$ thanks to \cref{L-FlatCyclic}.
Moreover, $Z=Z_1\oplus Z_2\leq \hat{F}$ and $\rho_i(\hat{F}_i)=\rho_i(Z_i)+\dim(\hat{F}_i/Z_i)$ by \cref{P-VcycVFormula}.
Now we compute
\begin{align*}
    M_1&=\rho_1(\hat{F}_1)+\rho_2(\hat{F}_2)-\dim(\hat{F}\cap V)\\
      &=\rho_1(Z_1)+\dim(\hat{F}_1/Z_1)+\rho_2(Z_2)+\dim(\hat{F}_2/Z_2)-\dim\hat{F}-\dim V+\dim(\hat{F}+V)\\
      &\geq\rho_1(Z_1)+\rho_2(Z_2)-\dim Z-\dim V+\dim(Z+V)\\
      &=\rho_1(Z_1)+\rho_2(Z_2)-\dim(Z\cap V)\geq M_2.
\end{align*}
Hence $M_1=M_2$, which establishes \eqref{e-rhoVZ}.
\end{proof}

The identity \eqref{e-rhoVZ} can be rewritten as the following convenient identity. Note its resemblance with \cref{C-RhoVFormula}.
Using this identity we will prove in the next section that $\cZ_1\oplus\cZ_2$ is the collection of cyclic flats of $\cM_1\oplus\cM_2$, and thus
it is indeed a special case of \cref{C-RhoVFormula}.

\begin{cor}\label{C-RhoDSZ}
In the situation of \cref{T-rhoVFlatsCyc} we have
\[
   \rho(V)=\min_{Z\in\cZ_1\oplus\cZ_2}\big(\rho(Z)+\dim(V+Z)/Z\big)\ \text{ for all }\ V\in\cL(E_1\oplus E_2).
\]
\end{cor}

\begin{proof}
Let $Z=Z_1\oplus Z_2\in\cZ_1\oplus\cZ_2$.  With the aid of \cref{T-DirSumProp}(a) we compute
\begin{align*}
   \rho_1(Z_1)+\rho_2(Z_2)-\dim(Z\cap V)&=\rho(Z)-\dim Z-\dim V+\dim(Z+V)\\
    &=\rho(Z)+\dim(Z+V)/Z-\dim V,
\end{align*}
and the result follows from \eqref{e-rhoVZ}.
\end{proof}

In the last part of this section we turn to the dual of the direct sum.
As shown next, it is the direct sum of the dual \qM{}s if taken with respect to compatible NSBFs.
Defining duality with respect to lattice anti-isomorphisms, the result below appears also in \cite[Thm.~50]{CeJu21}.
We will comment on the relation in \cref{R-LattAI-NSBF}.

\begin{theo}\label{T-DirSumDual}
Let $\cM_i=(E_i,\rho_i),\,i=1,2,$ be \qM{}s. Set $E=E_1\oplus E_2$.
Choose NSBFs $\inner{\cdot}{\cdot}_i$ on~$E_i$ and set
$\inner{v_1+v_2}{w_1+w_2}=\inner{v_1}{w_1}_1+\inner{v_2}{w_2}_2\ \text{ for all }v_i,w_i\in E_i$.
Then $\inner{\cdot}{\cdot}$ is an NSBF on~$E$ and
\[
    (\cM_1\oplus\cM_2)^*=\cM_1^*\oplus\cM_2^*,
\]
where $\cM_i^*$ and $(\cM_1\oplus\cM_2)^*$ are the dual $q$-matroids with respect to the given NSBFs.
\end{theo}

\begin{proof}
It is easy to see that $\inner{\cdot}{\cdot}$ is an NSBF on~$E$.
We denote the corresponding orthogonal space of $V\leq E$ by $V^\perp$, while
for $W\leq E_i$ we use $W^{\perp(i)}$ for the orthogonal of~$W$ in~$E_i$ with respect to $\inner{\cdot}{\cdot}_i$.
By construction we have
\[
   (V_1\oplus V_2)^\perp=V_1^{\perp(1)}\oplus V_2^{\perp(2)}\ \text{ for all }\ V_i\in\cL(E_i).
\]
Let $\cM_1^*\oplus\cM_2^*=(E,\tilde{\rho})$ and $\cM_1\oplus\cM_2=(E,\rho)$.
Then $(\cM_1\oplus\cM_2)^*=(E,\rho^*)$ with $\rho^*$ as in \eqref{e-rhodual}.
We have to show that $\rho^*(V)=\tilde{\rho}(V)$ for all $V\in\cL(E)$.
We will use  \eqref{e-rhoVT} for the rank function~$\rho$ (\cref{C-RhoDSZ} does not simplify the computations.).
Writing $X=X_1\oplus X_2$ for $X\in\cT$ we have
\[
   \tilde{\rho}(V)=\dim V+\min_{X\in\cT}\big(\rho_1^*(X_1)+\rho_2^*(X_2)-\dim(X\cap V)\big).
\]
With the aid of \cref{T-DirSumProp}(a) we obtain for any $X=X_1\oplus X_2\in\cT$
\begin{align*}
   &\rho_1^*(X_1)+\rho_2^*(X_2)-\dim(X\cap V)\\
       & =\dim X_1-\rho_1(E_1)+\rho_1(X_1^{\perp(1)})+\dim X_2-\rho_2(E_2)+\rho_2(X_2^{\perp(2)})-\dim(X\cap V)\\
       & =\dim X-\rho(E)+\rho(X^\perp)-\dim E+\dim(X\cap V)^\perp\\
       & =\dim X-\rho(E)+\rho(X^\perp)-\dim E+\dim(X^\perp+V^\perp)\\
        & =\dim X-\rho(E)+\rho(X^\perp)-\dim E+\dim X^\perp+\dim V^\perp-\dim(X^\perp\cap V^\perp)\\
        & =\dim V^\perp-\rho(E)+\rho(X^\perp)-\dim(X^\perp\cap V^\perp).
\end{align*}
Using $\{X^\perp\mid X\in\cT\}=\cT$ and again \eqref{e-rhoVT} we now arrive at
\begin{align*}
   \tilde{\rho}(V)&=\dim E-\rho(E)+\min_{X\in\cT}\big(\rho(X^\perp)-\dim(X^\perp\cap V^\perp)\big)\\
        &=\dim E-\rho(E)+\min_{X\in\cT}\big(\rho(X)-\dim(X\cap V^\perp)\big)\\
        &=\dim V-\rho(E)+\dim V^\perp+\min_{X\in\cT}\big(\rho_1(X_1)+\rho_2(X_2)-\dim(X\cap V^\perp)\big)\\
        &=\dim V-\rho(E)+\rho(V^\perp)=\rho^*(V),
\end{align*}
as desired.
\end{proof}

\begin{rem}\label{R-LattAI-NSBF}
In \cite[Def.~6]{CeJu21} the authors define duality of \qM{}s with respect to an involutory anti-isomorphism on the subspace lattice $\cL(E)$.
Denoting such an anti-isomorphism by $\perp$, the definition of $\cM^*$ reads exactly as in \cref{T-DualqM}.
Since the orthogonal complement with respect to a chosen NSBF induces a lattice anti-isomorphism, the duality result in \cite[Thm.~50]{CeJu21} appears
to be more general than \cref{T-DirSumDual}.
However, as we now briefly discuss, the two results differ only by a semi-linear isomorphism on~$E$ (if $\dim E\geq3$).
Indeed, choose an NSBF on~$E$ and denote the corresponding orthogonal space of~$V\leq E$ by
$V^{\raisebox{.2ex}{\mbox{$\scriptstyle\perp\hspace*{-.55em}\underline{\hspace*{.52em}}$}}}\hspace*{.2em}$.
Then $\tau:\cL(E)\longrightarrow\cL(E),\ V\longmapsto V^{\raisebox{.2ex}{\mbox{$\scriptstyle\perp\hspace*{-.55em}\underline{\hspace*{.52em}}$}}}\hspace*{.2em}$ is an anti-isomorphism on the lattice $\cL(E)$.
Let now $\perp$ be any anti-isomorphism on $\cL(E)$. Then $\tau\circ\!\perp$ is a lattice isomorphism and thanks to the Fundamental
Theorem of Projective Geometry (see for instance \cite[Ch.~II.10]{Art57} or \cite[Thm.~1]{Put})
there exists a semi-linear isomorphism $f:E\longrightarrow E$ such that $\tau(V^\perp)=f(V)$ for all $V\in\cL(E)$.
In other words,
 $(V^\perp)^{\raisebox{.2ex}{\mbox{$\scriptstyle\perp\hspace*{-.55em}\underline{\hspace*{.52em}}$}}}\hspace*{.2em}=f(V)$ or
$V^\perp=f(V)^{\raisebox{.2ex}{\mbox{$\scriptstyle\perp\hspace*{-.55em}\underline{\hspace*{.52em}}$}}}\hspace*{.2em}$ for all $V\in\cL(E)$.
This shows that the lattice anti-isomorphism $\perp$ differs from the one induced by the chosen NSBF by the semi-linear isomorphism on~$E$.
Denote the dual rank function of~$\cM=(E,\rho)$ with respect to $\perp$ and $\tau$ by
$\rho^{*(\perp)}$ and $\rho^{*(\raisebox{.2ex}{\mbox{$\scriptstyle\perp\hspace*{-.55em}\underline{\hspace*{.52em}}$}})}$,
respectively.
Then
$\rho^{*(\perp)}(V)=\rho^{*(\raisebox{.2ex}{\mbox{$\scriptstyle\perp\hspace*{-.55em}\underline{\hspace*{.52em}}$}})}(f(V))$ and thus the two
dual \qM{}s differ only by the semi-linear isomorphism~$f$.
This shows that \cite[Thm.~50]{CeJu21} is a consequence \cref{T-DirSumDual} above, which has a significantly shorter and simpler proof.
\end{rem}

\section{The Cyclic Flats of the Direct Sum}\label{S-DSProperties}
In this short section we show that the cyclic flats of a direct sum $\cM_1\oplus\cM_2$ is the collection of the direct sums of the cyclic flats of the two components~$\cM_1$ and~$\cM_2$.
The following lemma is needed.

\begin{lemma}\label{L-Z2}\
\begin{alphalist}
\item Let $\cM=(E,\rho)$ be a \qM{}. Suppose $F\in\cF(\cM)$ and $O\in\cO(\cM)$ are such that $F\lneq O$. Then
         $0<\rho(O)-\rho(F)<\dim O-\dim F$.
\item Let $\cM_i=(E_i,\rho_i),\,i=1,2,$ be \qM{}s and~$\rho$ be the rank function of $\cM_1\oplus\cM_2$.
         Let $F=F_1\oplus F_2$ with $F_i\in\cF(\cM_i)$ and $O=O_1\oplus O_2$ with $O_i\in\cO(\cM_i)$
         be such that $F\lneq O$.
         Then $0<\rho(O)-\rho(F)<\dim O-\dim F$.
\end{alphalist}
\end{lemma}

\begin{proof}
Since $F$ is a flat, we clearly have $0<\rho(O)-\rho(F)$.
Furthermore, let $O=F\oplus T$ for some $T\leq O$. Then (R1)--(R3) for~$\rho$ imply
$\rho(O)\leq\rho(F)+\rho(T)\leq\rho(F)+\dim T=\rho(F)+\dim O-\dim F$.
We show that the first inequality is strict.
To do so, let $U\in\Hyp(T)$.
Then $F\oplus U\in\Hyp(O)$.
Thus cyclicity of~$O$ implies $\rho(O)=\rho(F\oplus U)\leq \rho(F)+\rho(U)<\rho(F)+\dim T$, as desired.
\\
(b) By assumption $F_i\leq O_i$ for $i=1,2$.
Without loss of generality we may assume $F_1\lneq O_1$.
With the aid of \cref{T-DirSumProp}(a) and Part~(a) we compute
\begin{align*}
   \rho(O)-\rho(F)&=\rho_1(O_1)-\rho_1(F_1)+\rho_2(O_2)-\rho_2(F_2)\\
   &<\dim O_1-\dim F_1+\dim O_2-\dim F_2=\dim O-\dim F.
\end{align*}
The second expression also shows that $\rho(O)-\rho(F)>0$.
\end{proof}

Now we are ready for our main result.

\begin{theo}\label{T-CycFlatsDS}
Let $\cM_i=(E_i,\rho_i),\,i=1,2,$ be \qM{}s and $\cZ_i=\cZ(\cM_i)$.
As in \cref{N-DSCollections} let $\cZ_1\oplus\cZ_2=\{Z_1\oplus Z_2\mid Z_i\in\cZ_i\}$.
Then
\[
       \cZ(\cM_1\oplus\cM_2)=\cZ_1\oplus\cZ_2.
\]
\end{theo}

\begin{proof}
``$\supseteq$'' Let $V\in\cZ_1\oplus\cZ_2$. Then $V=V_1\oplus V_2$ for some $V_i\in\cZ_i$.
\\
a) We show that $V$ is a flat in $\cM_1\oplus\cM_2$. Let $x\in E\setminus V$.
We need to show that $\rho(V+\subspace{x})=\rho(V)+1$.
By \cref{C-RhoDSZ} there exists $\hat{Z}=\hat{Z}_1\oplus\hat{Z}_2\in\cZ_1\oplus\cZ_2$ such that
\[
   \rho(V+\subspace{x})=\rho(\hat{Z})+\dim((V+\subspace{x}+\hat{Z})/\hat{Z}).
\]
If $\hat{Z}=V$, then this implies $\rho(V+\subspace{x})=\rho(V)+1$, as desired.
Let now $\hat{Z}\neq V$ and set $F:=V\cap\hat{Z}=(V_1\cap\hat{Z}_1)\oplus(V_2\cap\hat{Z}_2)$.
Since each $V_i\cap\hat{Z}_i$ is a flat (see \cref{T-Basics}),~$F$ is of the form $F=F_1\oplus F_2$, where $F_i\in\cF(\cM_i)$.
By assumption~$V_i$ are cyclic spaces, and thus we may apply \cref{L-Z2}(b) to $F\lneq V$.
This leads to
\begin{equation}\label{e-rhoVZhat}
   \rho(V)-\rho(\hat{Z})\leq \rho(V)-\rho(F)<\dim V-\dim(F)=\dim(V/F),
\end{equation}
which in turn implies
\[
  \rho(V+\subspace{x})=\rho(\hat{Z})+\dim((V+\subspace{x}+\hat{Z})/\hat{Z})
     >\rho(V)-\dim(V/F)+\dim((V+\subspace{x}+\hat{Z})/\hat{Z}).
\]
Since $\dim(V/F)=\dim(V/(V\cap\hat{Z}))=\dim((V+\hat{Z})/\hat{Z})\leq\dim((V+\subspace{x}+\hat{Z})/\hat{Z})$, we conclude that
$\rho(V+\subspace{x})>\rho(V)$, as desired.
\\
b) We show that~$V$ is cyclic. Let $D\in\Hyp(V)$. By \cref{C-RhoDSZ} there exists $\hat{Z} \in\cZ_1\oplus\cZ_2$ such that
\[
   \rho(D)=\rho(\hat{Z})+\dim((D+\hat{Z})/\hat{Z}).
\]
If $\hat{Z}=V$, then this implies $\rho(D)=\rho(V)$, as desired. Thus let $\hat{Z}\neq V$.
As above, we set $F:=V\cap\hat{Z}$ and apply \cref{L-Z2}(b) to $F\lneq V$.
Thus we have again \eqref{e-rhoVZhat} and compute
\begin{align*}
  \rho(D)&=\rho(\hat{Z})+\dim((D+\hat{Z})/\hat{Z})\\
     &>\rho(V)-\dim(V/V\cap\hat{Z})+\dim(D/D\cap\hat{Z})=\rho(V)-1+\dim(V\cap\hat{Z})-\dim(D\cap\hat{Z})\\
    &\geq\rho(V)-1.
\end{align*}
This shows $\rho(D)=\rho(V)$.
\\
``$\subseteq$'' Let $Z\in\cZ(\cM_1\oplus\cM_2)$. Again, by \cref{C-RhoDSZ} there exists $\hat{Z}\in\cZ_1\oplus\cZ_2$ such that
\begin{equation}\label{e-rhoZZhat}
   \rho(Z)=\rho(\hat{Z})+\dim((Z+\hat{Z})/\hat{Z}).
\end{equation}
Since $Z$ is cyclic, every space $D\in\Hyp(Z)$ satisfies
\[
   \rho(D)=\rho(Z)=\rho(\hat{Z})+\dim((Z+\hat{Z})/\hat{Z})\geq\rho(\hat{Z})+\dim((D+\hat{Z})/\hat{Z})\geq\rho(D),
\]
and hence $\dim((D+\hat{Z})/\hat{Z})=\dim((Z+\hat{Z})/\hat{Z})$. Since this is true for every $D\in\Hyp(Z)$, we conclude $Z\leq\hat{Z}$.
Next,~$Z$ is a flat and thus every $x\in E\setminus Z$ satisfies
\[
  \rho(Z)<\rho(Z+\subspace{x})\leq\rho(\hat{Z})+\dim((Z+\subspace{x}+\hat{Z})/\hat{Z}).
\]
Together with \eqref{e-rhoZZhat} this implies $\dim(Z+\hat{Z})<\dim(Z+\subspace{x}+\hat{Z})$, and thus $x\not\in\hat{Z}$.
Since this is true for every $x\in E\setminus Z$, we conclude that $\hat{Z}\leq Z$.
All of this shows that $Z=\hat{Z}$ and thus $Z\in\cZ_1\oplus\cZ_2$.
This concludes the proof.
\end{proof}

\cref{T-CycFlatsDS} in combination with  \cref{C-RhoDSZ} immediately implies
associativity of the direct sum operation (which is not obvious from the very definition of the direct sum).
The result will be crucial for the decomposition of \qM{}s in the next section.

\begin{cor}\label{C-DSAssoc}
Let $\cM_i,\,i=1,2,3$, be \qM{}s. Then $(\cM_1\oplus\cM_2)\oplus\cM_3=\cM_1\oplus(\cM_2\oplus\cM_3)$.
\end{cor}

The analogous identity as in \cref{T-CycFlatsDS} is not true for the flats, independent spaces, circuits etc.

\begin{rem}\label{R-DSPropSpaces}
Let $\cM_i=(E_i,\rho_i),\,i=1,2,$ and $\cM=\cM_1\oplus\cM_2$. Recall \cref{N-DSCollections} and
denote by $\cI(\,\cdot\,)$ the collection of independent spaces of a \qM.
With the aid of \cref{T-DirSumProp}  it is easy to see that
$\cI(\cM_1)\oplus\cI(\cM_2)\subset\cI(\cM),\ \;  \cF(\cM_1)\oplus\cF(\cM_2)\subset\cF(\cM),\ \; \cO(\cM_1)\oplus\cO(\cM_2)\subset\cO(\cM)$,
and $\cC(\cM_i)\subset\cC(\cM)$ for $i=1,2$. In general equality does not hold in any of these cases.
\end{rem}

We illustrate the discrepancies in the following example.

\begin{exa}\label{E-F27}
Let $\F=\F_2$ and consider $\F_{2^3}$ with primitive element $\omega$ satisfying $\omega^3+\omega+1=0$.
Let
\[
    G_1=\begin{pmatrix}1&0&\omega^3\\0&1&\omega\end{pmatrix},\
    G_2=\begin{pmatrix}1&0&\omega^3&\omega\\0&1&\omega^4&\omega^2\end{pmatrix},\
    G=\begin{pmatrix}1&0&\omega^3&0&0&0&0\\0&1&\omega&0&0&0&0\\
                      0&0&0&1&0&\omega^3&\omega\\0&0&0&0&1&\omega^4&\omega^2\end{pmatrix}.
\]
Note that~$G$ is the block diagonal matrix with diagonal blocks~$G_1$ and~$G_2$.
Let $\cM_i=\cM_{G_i}$ and $\cN=\cM_G$, i.e., they are the \qM{}s represented by~$G_1,\,G_2$, and~$G$, respectively (see \cref{D-ReprG}).
Furthermore, let $\cM=\cM_1\oplus\cM_2$.
Thus both~$\cM$ and~$\cN$ have ground space~$\F^7$.
In the following table we list the number of flats, cyclic spaces,  etc. of all these \qM{}s.
\[
\begin{array}{|c||c|c|c|c|c|c|c|}
\hline
            &\text{flats}&\text{cyclic spaces}&\text{cyclic flats}&\text{ind.\ spaces}&\text{dep.\ spaces}&\text{circuits}&\text{bases}\\ \hline\hline
\rule[-.15cm]{0cm}{0.63cm}\cM_1 &  7 & 2 & 2 & 14 & 2 & 1 & 6\\ \hline
\rule[-.15cm]{0cm}{0.63cm}\cM_2& 11 & 11 & 5 &48 & 19 & 9 & 32\\ \hline
\rule[-.15cm]{0cm}{0.63cm} \cN&  2201 & 124 &40 & 24108 & 5104 & 73 & 9792\\ \hline
\rule[-.15cm]{0cm}{0.63cm}\cM & 7541 & 412 & 10 & 24861 & 4351 & 355 & 10416\\ \hline
\end{array}
\]
Note that the 10 cyclic flats of~$\cM$ are consistent with \cref{T-CycFlatsDS}.
It is remarkable that~$\cM$ has significantly more flats and cyclic spaces than $\cN$, yet fewer cyclic flats.
Furthermore, one can verify that the cyclic flats of~$\cM$ are also cyclic flats of $\cN$.
We will prove this in generality in \cref{T-CycFlatsDSGG} below.
Finally, all independent spaces of~$\cN$ are also independent spaces of~$\cM$.
This also follows from \cref{T-CycFlatsDSGG}.
\end{exa}

Our last result in this section concerns representable \qM{}s $\cM_1,\,\cM_2$.
It is well known that the direct sum of such \qM{}s is in general not representable; see \cite{GLJ22Rep}.
In particular, in contrast to the matroid case, the block diagonal matrix built from representing matrices~$G_1$ and~$G_2$
represents in general a \qM{}~$\cN$ that is different from $\cM_1\oplus\cM_2$.
The following result shows that the cyclic flats of $\cM_1\oplus\cM_2$ are also cyclic flats of~$\cN$.

\begin{theo}\label{T-CycFlatsDSGG}
Let $\F_{q^m}$ be a field extension of~$\F=\F_q$ and $G_i\in\F_{q^m}^{a_i\times n_i},\,i=1,2,$ be matrices of full row rank.
Set
\[
    G=\begin{pmatrix}G_1&0\\0&G_2\end{pmatrix}\in\F_{q^m}^{(a_1+a_2)\times(n_1+n_2)}.
\]
Denote by $\cM_i=(\F^{n_i},\rho_i),\,i=1,2,$ and $\cN=(\F^{n_1+n_2},\hat{\rho})$ the $q$-matroids represented by $G_1,G_2$, and~$G$, thus
$\rho_i(\rowsp(Y))=\rk(G_iY\T)$ for $Y\in\F^{y\times n_i}$ and $\hat{\rho}(\rowsp(Y))=\rk(GY\T)$ for $Y\in\F^{y\times(n_1+n_2)}$.
\begin{alphalist}
\item If $F\in\cF(\cM_1)\oplus\cF(\cM_2)$, then $F\in\cF(\cN)$.
\item If $O\in\cO(\cM_1)\oplus\cO(\cM_2)$, then $O\in\cO(\cN)$.
\end{alphalist}
As a consequence, $\cZ(\cM_1\oplus\cM_2)\subseteq\cZ(\cN)$ and thus $\cI(\cN)\subseteq\cI(\cM_1\oplus\cM_2)$.
\end{theo}

Note that the very last statement about the independent spaces also follows from the coproduct property of $\cM_1\oplus\cM_2$, discussed in \cref{T-DirSumCoPr}.

\begin{proof}
The first part of the consequence follows from \cref{T-CycFlatsDS} and the second part from  \cref{T-ZDeterminesM}.
For the proof of~(a) and~(b) we will make use of the fact
\begin{equation}\label{e-rho12}
   \hat{\rho}(V_1\oplus V_2)=\rho_1(V_1)+\rho_2(V_2)\ \text{ for all }V_i\in\cL(\F^{n_i}).
\end{equation}
This follows directly from the block diagonal form of~$G$ and the fact that every space $V_1\oplus V_2$ is the row space of a block diagonal matrix as well.
\\
(a) Let $F\in\cF(\cM_1)\oplus\cF(\cM_2)$ and $x\in\F^{n_1+n_2}\setminus F$.
We want to show that $\hat{\rho}(F+\subspace{x})>\hat{\rho}(F)$.
By assumption,
\[
    F=\rowsp\begin{pmatrix}Y_1&0\\0&Y_2\end{pmatrix}\ \text{ for some $Y_i\in\F^{y_i\times n_i}$ such that $\rowsp(Y_i)\in\cF(\cM_i)$.}
\]
Write $x=(x_1\mid x_2)$ with $x_i\in\F^{n_i}$. Without loss of generality let $x_1\not\in\rowsp(Y_1)$.
Since $\rowsp(Y_1)$ is a flat in $\cM_1$, this implies
\[
    \rk(G_1[Y_1^{\sf T},\,x_1^{\sf T}])=\rho_1(\rowsp(Y_1)+\subspace{x_1})>\rho_1(\rowsp(Y_1))=\rk (G_1 Y_1^{\sf T}).
\]
Hence $G_1x_1^{\sf T}$ is not in the column space of $G_1 Y_1^{\sf T}$ and we obtain
\begin{align*}
  \hat{\rho}(F+\subspace{x})=\rk\hspace*{-.3em}\left[\begin{pmatrix}G_1\!\!&0\\0\!\!&G_2\end{pmatrix}\!\!\!
               \begin{pmatrix}Y_1\!\!&0\\0\!\!&Y_2\\x_1\!\!&x_2\end{pmatrix}^{\hspace*{-.4em}\sf T}\,\right]\!\!=
    \rk\hspace*{-.3em}\begin{pmatrix}G_1Y_1^{\sf T}\!\!\!\!&0\!\!\!&G_1x_1^{\sf T}\\0\!\!\!\!&G_2Y_2^{\sf T}\!\!\!&G_2x_2^{\sf T}\end{pmatrix}
      > \rk\hspace*{-.3em}\begin{pmatrix}G_1Y_1^{\sf T}\!\!\!\!&0\\0\!\!\!\!&G_2Y_2^{\sf T}\end{pmatrix}=\hat{\rho}(F).
\end{align*}
Since~$x$ was arbitrary, this proves that~$F$ is a flat of~$\cN$.
\\
(b) Let $O\in\cO(\cM_1)\oplus\cO(\cM_2)$ and let $D\in\Hyp(O)$. We want to show that $\hat{\rho}(D)=\hat{\rho}(O)$.
By assumption $O=O_1\oplus O_2$ with $O_i\in\cO(\cM_i)$.
Recall the projections $\pi_i$ from $\F^{n_1+n_2}$ to $\F^{n_i}$ and set $\hat{D}=\pi_1(D)\oplus\pi_2(D)$.
Then $D\leq\hat{D}\leq O$ and since $\dim D=\dim O-1$, we have $\hat{D}=D$ or $\hat{D}=O$.
\\
(b1) If $\hat{D}=D$ we may assume $\pi_1(D)=O_1$ and $\pi_2(D)\in\Hyp(O_2)$.
Using cyclicity of~$O_2$ and \eqref{e-rho12} we arrive at
\[
   \hat{\rho}(D)=\hat{\rho}(\hat{D})=\rho_1(\pi_1(D))+\rho_2(\pi_2(D))=\rho_1(O_1)+\rho_2(O_2)=\hat{\rho}(O),
\]
which is what we wanted.
\\
(b2) Let $\hat{D}=O$. Set $\dim O=k+1$ and thus $\dim D=k$. Write $D=\rowsp(M_1\mid M_2)$, where $M_i\in\F^{k\times n_i}$. Then
\[
     \hat{D}=O=\rowsp\begin{pmatrix}M_1&0\\0&M_2\end{pmatrix}.
\]
Hence $O_i=\rowsp(M_i)$. Let $\rk M_1=k_1$. Then $\rk M_2=k-k_1+1$.
Using elementary row operations we may assume that~$(M_1\mid M_2)$ is of the form
\[
    \begin{pmatrix}m_{1}&m_{2}\\M_{21}&0\\0&M_{22}\end{pmatrix}\ \text{ for some }
    m_{i}\in\F^{n_i},\ M_{21}\in\F^{(k_1-1)\times n_1},\ M_{22}\in\F^{(k-k_1)\times n_2}.
\]
Since $\rk M_{2i}=\rk M_i-1$ it follows that $\rowsp(M_{2i})\in\Hyp(\rowsp(M_i))$.
Using the cyclicity of $\rowsp(M_i)$ we conclude that $\rk(G_iM_{2i}^{\sf T})=\rk(G_iM_i^{\sf T})=\rk(G_i[m_i^{\sf T},M_{2i}^{\sf T}])$ for $i=1,2$.
This means that $G_im_i^{\sf T}$ is in the column space of $G_iM_{2i}^{\sf T}$, and thus
\[
   \hat{\rho}(D)=\rk\hspace*{-.3em}\begin{pmatrix}G_1m_1^{\sf T}&G_1M_{21}^{\sf T}&0\\G_2m_2^{\sf T}&0&G_2M_{22}^{\sf T}\end{pmatrix}
   =\rk\hspace*{-.3em}\begin{pmatrix}G_1M_{21}^{\sf T}&0\\0&G_2M_{22}^{\sf T}\end{pmatrix}
    =\rk\hspace*{-.3em}\begin{pmatrix}G_1M_1^{\sf T}&0\\0&G_2M_2^{\sf T}\end{pmatrix}=\hat{\rho}(O).
\]
All of this shows $\hat{\rho}(D)=\hat{\rho}(O)$ for every $D\in\Hyp(O)$ and thus $O\in\cO(\cN)$.
This concludes the proof.
\end{proof}

\section{Decomposition of $q$-Matroids into Irreducible Components}\label{S-Decomp}
We introduce the notion of irreducibility for \qM{}s and show that every \qM{} can be decomposed as a
direct sum of irreducible \qM{}s, whose summands are unique up to equivalence.
Our main tool are cyclic flats, in particular \cref{T-CycFlatsDS}.
This makes our approach substantially different from classical matroid theory, where decompositions are usually based on connected components.
As to our knowledge there is no notion of connectedness for \qM{}s that may be used for decompositions into direct sums; see also \cite[Sec.~8]{CeJu21}.

Throughout, let $\cM=(E,\rho)$ be a \qM{}.
In order to simplify the discussion of irreducibility and decompositions we start with the following simple fact concerning equivalence in the
sense of~\cref{D-Equiv}. It can easily be checked with the definition in \cref{T-DirSum}.

\begin{rem}\label{R-SimpleFact}
Suppose $\cM\approx\hat{\cM}_1\oplus\hat{\cM}_2$. Then there exists a decomposition $E=E_1\oplus E_2$ and \qM{}s $\cM_i$ such that
\[
   \cM=\cM_1\oplus\cM_2,\quad \cM_i=\cM|_{E_i},\quad \cM_i\approx\hat{\cM}_i.
\]
As a consequence, we do not need to take equivalence into account when discussing decomposability into direct sums.
\end{rem}

\begin{defi}\label{D-Irred}
The \qM{}~$\cM$ is called \emph{reducible} if there exists \qM{}s $\cM_1,\,\cM_2$ with nonzero ground spaces such that
$\cM=\cM_1\oplus\cM_2$. Otherwise~$\cM$ is called \emph{irreducible}.
\end{defi}

Clearly,  a \qM{} on a 1-dimensional ground space is irreducible.
Furthermore, thanks to \cref{T-DirSumDual} $\cM$ is irreducible if and only if~$\cM^*$ is.
We collect some facts about the uniform \qM{}s.

\begin{exa}\label{E-TrivialFreeRed}
\begin{alphalist}
\item The trivial and the free \qM{}s $\cU_{0,n}$ and $\cU_{n,n}$ are irreducible if and only if $n=1$.
Indeed, \cref{T-DirSum} implies $\cU_{0,n_1}\oplus\cU_{0,n_2}=\cU_{0,n}$ and likewise
$\cU_{n_1,n_1}\oplus\cU_{n_2,n_2}=\cU_{n,n}$, where $n=n_1+n_2$.
\item For $0<k<n:=\dim E$ the uniform \qM{} $\cM:=\cU_k(E)$  is irreducible.
To see this, note first that $\cZ(\cM)=\{0,E\}$ thanks to \cref{E-UniMatr}.
Suppose $\cM=\cM_1\oplus\cM_2$ for some \qM{}s $\cM_i=(E_i,\rho_i)$.
Then the identity  $\cZ(\cM)=\cZ(\cM_1)\oplus\cZ(\cM_2)$ implies that, without loss of generality, $\cZ(\cM_1)=\{Z_1\}$ and
$\cZ(\cM_2)=\{Z_2,Z_2'\}$.
Thus $0=Z_1\oplus Z_2$ and $E=Z_1\oplus Z_2'$, and therefore $Z_1=0,\,Z_2=0,\,Z_2'=E$.
Since $E=E_1\oplus E_2$ and $Z_2'\leq E_2$, this leads to $E_2=E$ and thus $E_1=0$.
Hence $\cM_1$ has a zero-dimensional ground space and~$\cM$ is irreducible.
\item Conversely, if $\cM=(E,\rho)$ is such that $\cZ(\cM)=\{0,E\}$ and $\rho(E)=k\in\{1,\ldots,n-1\}$, then $\cM=\cU_k(E)$.
Indeed, suppose there exists a space~$V$ such that $\rho(V)<\min\{k,\dim V\}$.
Let~$l=\dim V$ be minimal subject to this condition. Then $\rho(V)=l-1$ and~$V$ is cyclic.
Thus $\cl(V)=E$ and $\rho(V)=\rho(E)=k$, contradicting the choice of~$V$.
This shows $\rho(V)=\min\{k,\dim V\}$ for all $V\in\cL(E)$.
\end{alphalist}
\end{exa}

The goal of this section is (a) to provide a criterion for irreducibility and (b) to show that every \qM{} decomposes into a direct
sum of irreducible \qM{}s, whose summands are unique up to ordering and equivalence.
The next lemma will be needed throughout.

\begin{lemma}\label{L-RestrToCF}
Let $\hat{Z}\in\cZ(\cM)$ and consider the restriction $\cM|_{\hat{Z}}$. Then
\[
    \cZ(\cM|_{\hat{Z}})=\{Z\in\cZ(\cM)\mid Z\leq \hat{Z}\}.
\]
\end{lemma}

\begin{proof}
Set $\cZ=\{Z\in\cZ(\cM)\mid Z\leq \hat{Z}\}$. Denote the rank function of $\cM|_{\hat{Z}}$ by~$\hat{\rho}$.
\\
``$\supseteq$'' is obvious since $\hat{\rho}(V)=\rho(V)$ for all $V\leq\hat{Z}$.
\\
``$\subseteq$'' Let $Z\in\cZ(\cM|_{\hat{Z}})$.
Clearly~$Z$ is cyclic in~$\cM$ because any $D\in\Hyp(Z)$ is a subspace of $\hat{Z}$ and thus satisfies $\rho(D)=\hat{\rho}(D)=\hat{\rho}(Z)=\rho(Z)$.
To show that $Z$ is a flat in~$\cM$, let $x\in E\setminus Z$.
If $x\in \hat{Z}\setminus Z$, then $\rho(Z+\subspace{x})=\hat{\rho}(Z+\subspace{x})>\hat{\rho}(Z)=\rho(Z)$.
If $x\in E\setminus \hat{Z}$, then $x\not\in\cl(Z)$ because $\cl(Z)\leq \cl(\hat{Z})=\hat{Z}$ (where $\cl(\,\cdot\,)$ denotes the closure in~$\cM$).
Thus $\rho(Z+\subspace{x})>\rho(Z)$.
All of this shows that~$Z$ is a flat in~$\cM$.
\end{proof}

Our first result shows that whenever $E$ is not a cyclic flat of~$\cM$, then for any direct complement~$E_2$ of $\cyc(E)$ in~$E$
we may split off the free \qM{} on~$E_2$ from~$\cM$.

\begin{prop}\label{P-FreeFactor}
Let $E_1=\cyc(E)$ and choose $E_2\leq E$ such that $E_1\oplus E_2=E$.
Consider the restrictions $\cM_i=\cM|_{E_i}$ for $i=1,2$. Then
\begin{alphalist}
\item $\cM_2$ is the free \qM{} on~$E_2$.
\item $\cZ(\cM)=\cZ(\cM_1)$.
\item $\cM=\cM_1\oplus\cM_2$.
\end{alphalist}
\end{prop}

\begin{proof}
(a)  \cref{P-VcycVFormula}(b) implies that~$E_2$ is independent in~$\cM$ and thus in~$\cM_2$.
Since every subspace of an independent space is independent, the result follows.
\\
(b) follows from \cref{L-RestrToCF} together with the fact that $\cZ(\cM)$ is a lattice with greatest element $\cyc(E)$.
\\
(c) Part (a) and \cref{E-UniMatr} tell us that $\cZ(\cM_2)=\{0\}$, and thus (b) and \cref{T-CycFlatsDS} imply
\[
   \cZ(\cM_1\oplus\cM_2)=\cZ(\cM_1)=\cZ(\cM).
\]
Denote the rank function of $\cM_1\oplus\cM_2$ by $\hat{\rho}$.
Then \cref{T-DirSumProp}(a) implies $\hat{\rho}(Z_1)=\rho_1(Z_1)=\rho(Z_1)$ for all $Z_1\in\cZ(\cM_1)$.
Hence the cyclic flats in $\cZ(\cM)$ have the same rank value in the \qM{}s  $\cM_1\oplus\cM_2$ and $\cM$.
The result follows from \cref{C-RhoVFormula}.
\end{proof}

Dually, we may split off the trivial \qM{} on $\cl(0)$ from~$\cM$.

\begin{prop}\label{P-TrivialFactor}
Let $E_1=\cl(0)$ and choose $E_2\leq E$ such that $E=E_1\oplus E_2$.
Consider the restrictions $\cM_i=\cM|_{E_i}=(E_i,\rho_i)$ for $i=1,2$ and let $\pi_i:E\longrightarrow E_i$ be the projections.
\begin{alphalist}
\item $\cM_1$ is the trivial \qM{} on $E_1$.
\item $\rho(V)=\rho_2(\pi_2(V))$ for all $V\in\cL(E)$.
\item $\cM=\cM_1\oplus\cM_2$.
\end{alphalist}
We call $\cl(0)$ the \emph{loop space} of~$\cM$. It consists of all vectors~$x\in E$ such that $\rho(\subspace{x})=0$.
\end{prop}

\begin{proof}
(a) is clear and so is the very last part about the vectors in the loop space.
\\
(b) Let $V\in\cL(E)$.
Since $V\leq\pi_1(V)\oplus\pi_2(V)$ we have $\rho(V)\leq \rho(\pi_1(V))+\rho(\pi_2(V))=\rho(\pi_2(V))=\rho_2(\pi_2(V))$.
For the converse inequality, let $y\in\pi_2(V)$. Then there exists $x\in E_1$ such that $x+y\in V$.
As a consequence, $V+\subspace{y}=V+\subspace{x}$ and $\rho(V)\leq\rho(V+\subspace{x})\leq\rho(V)+\rho(\subspace{x})=\rho(V)$.
Hence we have equality across and this shows that $\rho(V+\subspace{y})=\rho(V)$ for all $y\in\pi_2(V)$.
Now \cref{P-RhoProp}(a) leads to $\rho(V+\pi_2(V))=\rho(V)$ and thus $\rho(\pi_2(V))\leq\rho(V)$.
\\
(c) Let $\hat{\rho}$ be the rank function of $\cM_1\oplus\cM_2$.
Then
\begin{align*}
  \hat{\rho}(V)&=\dim V+\min_{X\leq V}\big(\rho_2(\pi_2(X))-\dim X\big)
     =\dim V+\min_{X\leq V}\big(\rho(X)-\dim X\big)\\
     &=\dim V+\rho(V)-\dim V=\rho(V),
\end{align*}
where the third step follows from the inequality $\dim A-\rho(A)\leq \dim B-\rho(B)$ for any $A\leq B$ (which is a simple consequence of submodularity).
\end{proof}

The following terminology will be convenient.

\begin{defi}\label{D-Full}
$\cM$ is called \emph{full} if $\cl(0)=0$ and $\cyc(E)=E$.
\end{defi}

\begin{rem}\label{R-Coloop}
The notion of a full matroid does not exist in classical matroid theory (as to our knowledge) because a matroid satisfying $\cl(0)=0$ and $\cyc(E)=E$
is simply called loopless and coloopless.
In the $q$-analogue, however, the dual \qM{} depends on the choice of a NSBF (see \cref{T-DualqM}) and therefore the notion of a coloop is not well-defined.
Indeed, a 1-dimensional subspace $\subspace{x}$ would be called a coloop of~$\cM$ if it is a loop of~$\cM^*$, that is, if
$\subspace{x}\leq\cl^*(0)$.
By \cref{C-cyccl} $\cl^*(0)=\cyc(E)^\perp$ and thus $\cl^*(0)\neq0\Longleftrightarrow\cyc(E)\neq E$.
But if $\cyc(E)\neq E$ one can find for any nonzero vector $x\in E$ a NSBF such that $x\in\cyc(E)^\perp$, hence $\subspace{x}$ is a loop in the
dual \qM{} with respect to this NSBF.
For this reason we will not use the notion of coloops.
\end{rem}

Now we can present a first step toward a decomposition of~$\cM$.

\begin{theo}\label{T-SplitOff}
Given $\cM=(E,\rho)$. Set $l=\dim\cl(0)$ and $f=\dim E-\dim\cyc(E)$.
\begin{alphalist}
\item $\cM$ is the direct sum of a trivial, a free, and a full \qM.
Precisely,  there exists a full \qM{} $\cM'$ such that
\[
   \cM\approx\cU_{0,l}\oplus\cU_{f,f}\oplus\cM'.
\]
\item If $\cM\approx\cU_{0,a}\oplus\cU_{b,b}\oplus\cN$, where~$\cN$ is full, then $a=l,\,b=f$ and $\cN\approx\cM'$.
\end{alphalist}
We call $\cU_{0,l},\,\cU_{f,f}$, and~$\cM'$ the \emph{trivial, free and full component of}~$\cM$, respectively.
\end{theo}

\begin{proof}
Let $E=\cyc(E)\oplus\Delta$ and $\cyc(E)=\cl(0)\oplus \Gamma$. Then $\dim\Delta=f$.
\\
(a) By \cref{P-FreeFactor} we have $\cM=\cU_{f}(\Delta)\oplus\cM|_{\cyc(E)}$, and
\cref{P-TrivialFactor} implies $\cM|_{\cyc(E)}=\cU_{0}(\cl(0))\oplus\cM|_{\Gamma}$.
This proves the stated decomposition of~$\cM$, and it remains to show that $\cM':=\cM|_{\Gamma}$ is full.
Again Propositions~\ref{P-FreeFactor} and \ref{P-TrivialFactor}
give us $\cZ(\cM)=\cZ(\cM|_{\cyc(E)})=\{\cl(0)\oplus Z\mid Z\in\cZ(\cM')\}$.
Since $\cl(0)$ and $\cyc(E)=\cl(0)\oplus\Gamma$ are the least and greatest element of the lattice $\cZ(\cM)$, we conclude that~$0$
and~$\Gamma$ are the least and greatest element of the lattice $\cZ(\cM')$.
Thus $\cM'$ is full.
\\
(b) Using \cref{R-SimpleFact} we have
\[
    \cM=\cU_0(A)\oplus\cU_b(B)\oplus\cN
\]
for some  $A,B,N\leq E$ such that $A\oplus B\oplus N=E,\, a=\dim A,\,b=\dim B$, and where~$N$ is the ground space of~$\cN$.
Moreover, by~(a)
\[
      \cM=\cU_{0}(\cl(0))\oplus\cU_f(\Delta)\oplus\cM'.
\]
From \cref{T-CycFlatsDS} and \cref{E-UniMatr} we obtain
$\cZ(\cM)=\{A\}\oplus\cZ(\cN)$.
Since~$\cN$ is full, the least element of the lattice $\cZ(\cN)$ is~$0$ and thus~$A$ is the least element of~$\cZ(\cM)$.
But the latter is $\cl(0)$ and thus we arrive at $A=\cl(0)$ and $a=l$.
In the same way, the greatest element of $\cZ(\cN)$ is~$N$ and thus $A\oplus N$ is the greatest element of~$\cZ(\cM)$.
Hence $A\oplus N=\cyc(E)$ and $a+\dim N=\dim\cyc(E)$, which implies $f=\dim E-\dim\cyc(E)=b$, as desired.
In order to show that $\cN\approx\cM'$ note that $\cyc(E)=A\oplus\Gamma=A\oplus N$ and
$\cM|_{\cyc(E)}=\cU_0(A)\oplus\cM'=\cU_0(A)\oplus\cN$ (see \cref{T-DirSumProp}(a)).
Now \cref{T-DirSumProp}(c) yields $\cN\approx(\cM|_{\cyc(E)})/A\approx\cM'$.
This concludes the proof.
\end{proof}

We have the following special case.

\begin{cor}\label{C-SingleCF}
$|\cZ(\cM)|=1$ if and only if~$\cM$ is the direct sum of a trivial and a free \qM.
Precisely, let  $l=\dim\cl(0)$ and $f=\dim E-l$.
Then $\cZ(\cM)=\{\cl(0)\}\Longleftrightarrow\cM\approx\cU_{0,l}\oplus\cU_{f,f}$.
\end{cor}

\begin{proof}
The backward direction follows immediately from \cref{T-CycFlatsDS} together with $|\cZ(\cU_{0,l})|=|\cZ(\cU_{f,f})|=1$ for any $l,f$.
The forward direction is a consequence of \cref{T-SplitOff} because the assumption implies $\cl(0)=\cyc(E)$.
\end{proof}

In order to derive a criterion for irreducibility we need the following two lemmas.
For any subspace $V\in\cL(E)$ we use $\cB(V)$ for the
collection of bases of~$V$, i.e., $\cB(V)=\{I\leq V\mid \dim I=\rho(I)=\rho(V)\}$.

\begin{lemma}\label{L-SumOfFlats}
Suppose there exist flats $F_1,\,F_2$ of~$\cM$ such that $F_1\cap F_2=0$ and $\rho(F_1\oplus F_2)=\rho(F_1)+\rho(F_2)$.
\begin{alphalist}
\item Let $B_i\in\cB(F_i)$. Then  $B_1\oplus B_2\in\cB(F_1\oplus F_2)$.
\item Let $V_i\in\cL(F_i)$. Then $\rho(V_1\oplus V_2)=\rho(V_1)+\rho(V_2)$.
\end{alphalist}
\end{lemma}

\begin{proof}
(a) $B_i\leq F_i=\cl(F_i)$ together with $\rho(B_i)=\rho(F_i)$ implies $\cl(B_i)=F_i$.
Since $\cl(B_i)\leq\cl(B_1\oplus B_2)$, this leads to
$B_1\oplus B_2\leq F_1\oplus F_2=\cl(B_1)\oplus\cl(B_2)\leq\cl(B_1\oplus B_2)$,
and thus $\rho(B_1\oplus B_2)=\rho(F_1\oplus F_2)$ thanks to \eqref{e-clrho}.
Now we have $\rho(B_1\oplus B_2)=\rho(F_1)+\rho(F_2)=\dim B_1+\dim B_2=\dim(B_1\oplus B_2)$, which shows that
$B_1\oplus B_2$ is a basis of $F_1\oplus F_2$.
\\
(b) Let $B_i\in\cB(V_i)$.
Then~$B_i$ is an independent space in~$F_i$ and thus contained in a basis of~$F_i$, say $B'_i$ (see \cite[Thm.~37]{JuPe18}).
Thanks to part~(a) $B'_1\oplus B'_2$ is independent and hence so is $B_1\oplus B_2$.
Putting everything together, we obtain
\[
  \rho(V_1\oplus V_2)\leq \rho(V_1)+\rho(V_2)=\dim B_1+\dim B_2=\dim(B_1\oplus B_2)=\rho(B_1\oplus B_2)\leq\rho(V_1\oplus V_2).
\]
which proves the stated identity.
\end{proof}

\begin{lemma}\label{L-FactorFull}
Let~$\cM$ be full. Suppose $\cM=\cM_1\oplus\cM_2$ for some \qM{}s $\cM_i=(E_i,\rho_i)$.
Then $E_i\in\cZ(\cM)$ and $\cM_1,\,\cM_2$ are full.
\end{lemma}

\begin{proof}
Since $\cM$ is full,~$E$ is in $\cZ(\cM)=\cZ(\cM_1)\oplus\cZ(\cM_2)$.
Thus there exist $Z_i\in\cZ(\cM_i)$ such that $E=Z_1\oplus Z_2$. But then $E=E_1\oplus E_2$ together with $Z_i\leq E_i$
implies $Z_i=E_i$,
and we conclude that~$E_i\in\cZ(\cM_i)$. In the same way, $0\in\cZ(\cM)$ implies $0\in\cZ(\cM_i)$ for $i=1,2$,
and hence $\cM_1$ and $\cM_2$ are full.
Finally, $E_1=E_1\oplus0$ is in $\cZ(\cM_1)\oplus\cZ(\cM_2)=\cZ(\cM)$ and similarly for $E_2$.
\end{proof}

Now we are ready for our first main result of this section, a characterization of irreducibility.

\begin{theo}\label{T-Irred}
Let $\cM=(E,\rho)$ and $\dim E\geq 2$. The following are equivalent.
\begin{romanlist}
\item $\cM$ is irreducible.
\item $\cM$ is full and there exist no nonzero spaces $Z_1,Z_2\in\cZ(\cM)$ such that
       \begin{equation}\label{e-CFlatCond}
             Z_1\oplus Z_2=E,\quad \rho(Z_1)+\rho(Z_2)=\rho(E),\quad  \text{ and }\ \cZ_1\oplus\cZ_2=\cZ(\cM),
       \end{equation}
       where $\cZ_i=\{Z\in\cZ(\cM)\mid Z\leq Z_i\}$.
\end{romanlist}
\end{theo}

Note that $\cZ_i=\cZ(\cM|_{Z_i})$ thanks to \cref{L-RestrToCF}.

\begin{proof}
``(i) $\Rightarrow$ (ii)'' Let~$\cM$ be irreducible.
Then \cref{E-TrivialFreeRed}(a) and \cref{T-SplitOff} together with $\dim E\geq 2$ imply that~$\cM$ is full.
Suppose there do exist nonzero spaces $Z_1,Z_2\in\cZ(\cM)$ satisfying \eqref{e-CFlatCond}. We show that
\[
    \cM=\cM|_{Z_1}\oplus\cM|_{Z_2}.
\]
Set $\cZ=\cZ(\cM|_{Z_1})\oplus\cZ(\cM|_{Z_2})$ and denote the rank functions of $\cM|_{Z_1}\oplus\cM|_{Z_2}$ and $\cM|_{Z_i}$  by~$\rho'$ and~$\rho_i$,
respectively.
With the aid of \cref{T-rhoVFlatsCyc} and \cref{L-SumOfFlats}(b) we obtain for all $V\in\cL(E)$
\begin{align*}
   \rho'(V)&=\dim V+\min_{Y_1\oplus Y_2\in\cZ}\Big(\rho_1(Y_1)+\rho_2(Y_2)-\dim((Y_1\oplus Y_2)\cap V)\Big)\\
     &=\dim V+\min_{Y_1\oplus Y_2\in\cZ}\Big(\rho(Y_1)+\rho(Y_2)-\dim((Y_1\oplus Y_2)\cap V)\Big)\\
     &=\min_{Y_1\oplus Y_2\in\cZ}\Big(\rho(Y_1\oplus Y_2)+\dim((V+(Y_1\oplus Y_2))/(Y_1\oplus Y_2))\Big)\\
     &=\rho(V),
\end{align*}
where the very last step follows from \cref{C-RhoVFormula} and the identity $\cZ=\cZ(\cM)$.
This establishes the stated direct sum and thus contradicts the irreducibility of~$\cM$.
\\
``(ii) $\Rightarrow$ (i)''
Suppose~$\cM$ is full.
By contradiction assume that~$\cM$ is reducible, say $\cM=\cM_1\oplus\cM_2$ for some $\cM_i=(E_i,\rho_i)$
with nonzero ground spaces~$E_i$.
Then  $E=E_1\oplus E_2$ and $\rho(E)=\rho(E_1)+\rho(E_2)$; see \cref{T-DirSumProp}(a).
Moreover, $\cM_i=\cM|_{E_i}$.
Now \cref{L-FactorFull} implies that $E_i\in\cZ(\cM)$ and
\cref{T-CycFlatsDS} shows that $\cZ(\cM)=\cZ(\cM|_{E_1})\oplus\cZ(\cM|_{E_2})$.
\cref{L-RestrToCF} tells us that $\cZ(\cM|_{E_i})=\{Z\in\cZ(\cM)\mid Z\leq E_i\}$, and
all of this gives us cyclic flats $E_1,\,E_2$ satisfying the conditions in~\eqref{e-CFlatCond}.
\end{proof}

Since the collection of cyclic flats is in general quite small, the just presented criterion for irreducibility is in fact very convenient.
For instance, simple inspection shows that the \qM{} in \cref{E-CycFlatsStart} is irreducible.

\begin{exa}\label{E-F27Cont}
Let us consider the \qM{}s $\cM_1,\,\cM_2,\,\cM=\cM_1\oplus\cM_2$, and~$\cN$ from \cref{E-F27}.
\\
(1) We start with~$\cM_2=\cM_{G_2}$, which has ground space~$\F^4$. Its cyclic flats are
\[
   0,\subspace{e_1+e_3,\,e_2+e_3+e_4},\ \subspace{e_1+e_3+e_4,\,e_2},\ \subspace{e_3,\,e_4},\ \F^4
\]
with rank values $0,1,1,1,2$. Hence $\cM_2$ is full and in fact irreducible.
\\
(2) Consider now $\cM_1=\cM_{G_1}$, which has ground space $\F^3$. Its cyclic flats are
\[
   0,\,\subspace{e_1+e_3,\,e_2+e_3}.
\]
Hence $\cM_1$ is not full and since $\cyc(\F^3)=\subspace{e_1+e_3,\,e_2+e_3}$, we can split off a
free \qM{} over a 1-dimensional ground space; see  \cref{T-SplitOff}.
The other summand is $\cM_1|_{\cyc(\F^3)}$.
Straightforward verification shows that this \qM{} is  the uniform
 \qM{} of rank 1 on~$\cyc(\F^3)$.
Since $\cU_1(\F^2)$ is the only full \qM{} on~$\F^2$ of rank~$1$, we obtain
$\cM_1\approx\cU_1(\F)\oplus\cU_1(\F^2)$.
\\
(3) As a consequence, the \qM{} $\cM$ has irreducible decomposition $\cM\approx\cU_1(\F)\oplus\cU_1(\F^2)\oplus\cM_2$.
\\
(4) As for the \qM{} $\cN=\cM_G$, we have $\cl(0)=0$ and $\dim\cyc(\F^7)=6$, thus $\cyc(\F^7)=\rowsp(A)$ for some $A\in\F^{6\times 7}$.
The \qM{} $\cN|_{\cyc(\F^7)}$ is equivalent to the \qM{} $\cN'$ represented by the matrix
$GA\T\in\F_{2^3}^{4\times 6}$, which has rank~$3$.
By construction $\cN'$ is full (see \cref{T-SplitOff}) and $\cN\approx\cU_1(\F)\oplus\cN'$.
Since~$\cU_1(\F)$ has exactly one cyclic flat,~$\cN'$ has $40$ cyclic flats just like~$\cN$; see \cref{T-CycFlatsDS}.
Apart from $\cl(0)=0$ and $\cyc(\F^6)=\F^6$ with rank 3, the cyclic flats of~$\cN'$  are as follows:
\begin{align*}
  &\text{\,\ 7 cyclic flats of dimension 2 and rank 1,}\\
  &\text{24 cyclic flats of dimension 3 and rank 2,}\\
  &\text{\,\ 7 cyclic flats of dimension 4 and rank 2.}
\end{align*}
In order to check~$\cN'$ for irreducibility it thus suffices to test whether any pair $(Z_1,Z_2)$ of cyclic flats with $\dim Z_1=2$ and $\dim Z_2=4$
satisfies \eqref{e-CFlatCond}.
It turns out that there are 28 pairs $(Z_1,Z_2)$ such that $Z_1\oplus Z_2=\F^6$.
By \cref{L-RestrToCF} the collection $\cZ(\cN'|_{Z_i})$ consists exactly of the cyclic flats contained in~$Z_i$.
Now one easily finds that for each pair $(Z_1,Z_2)$ one has $|\cZ(\cN'|_{Z_1})|=2$ and $|\cZ(\cN'|_{Z_2})|=5$, and thus the third condition of~\eqref{e-CFlatCond} is not satisfied.
All of this shows that $\cN\approx\cU_1(\F)\oplus\cN'$ is a decomposition of~$\cN$ into irreducible \qM{}s.
\end{exa}

We continue with our second main result, the decomposition of \qM{}s into irreducible summands.

\begin{theo}\label{T-UniqueDecomp}
Any \qM{} $\cM=(E,\rho)$ is a direct sum of irreducible \qM{}s, whose summands are unique up to equivalence.
\end{theo}

\begin{proof}
It is clear that every \qM{} is the direct sum of irreducible \qM{}s.
For the uniqueness, note first that thanks to \cref{T-SplitOff} we may disregard trivial and free summands.
Thus we may assume that~$\cM$ is full and, using once more \cref{R-SimpleFact}, that
\begin{equation}\label{e-TwoDecomp}
   \cM=\bigoplus_{i=1}^t\cM_i=\bigoplus_{i=1}^s\cN_i,
\end{equation}
where each summand~$\cM_i$ and $\cN_i$ is  irreducible (and full by \cref{L-FactorFull}).
We will show that $s=t$ and, after suitable ordering, $\cN_i=\cM_i$ for all $i=1,\ldots, t$.
Thanks to \cref{L-FactorFull} we have
\[
   \cM_i=(Z_i,\rho_i),\ \cN_i=(Z'_i,\rho'_i),\ \text{ where }\
  E=\bigoplus_{i=1}^tZ_i=\bigoplus_{i=1}^s Z'_i\ \text{ and }\ Z_i,\,Z'_i\in\cZ(\cM).
\]
Thus $\cM_i=\cM|_{Z_i}$ and $\cN_i=\cM|_{Z'_i}$ for all~$i$.
Moreover, $\cZ(\cM)=\bigoplus_{i=1}^t\cZ(\cM_i)=\bigoplus_{i=1}^s\cZ(\cN_i)$.
\\
Consider now $Z'_1$. Then
\begin{equation}\label{e-Z'1Sum}
       Z'_1=\bigoplus_{i=1}^t V_i\ \text{ for some }V_i\in\cZ(\cM_i).
\end{equation}
Hence $V_i\leq Z_i$ for all~$i$ and by \cref{T-DirSumProp}(a)
\begin{equation}\label{e-RhoSum}
   \rho(Z'_1)=\sum_{i=1}^t\rho_i(V_i)=\sum_{i=1}^t\rho(V_i).
\end{equation}
We show next that
\begin{equation}\label{e-CFIdentity}
      \cZ(\cN_1)=\bigoplus_{i=1}^t\cZ(V_i), \text{ where }\ \cZ(V_i)=\{Z\in\cZ(\cM_i)\mid Z\leq V_i\}.
\end{equation}
For ``$\supseteq$'' choose $W_i\in\cZ(V_i)$. Then $W:=\bigoplus_{i=1}^t W_i\leq\bigoplus_{i=1}^t V_i=Z'_1$.
Furthermore,~$W$ is in $\cZ(\cM)$, and thus $W\in\{Z\in\cZ(\cM)\mid Z\leq V_i\}\subseteq\cZ(\cN_1)$, where the last containment
is a consequence of \cref{L-RestrToCF}.
\\
For ``$\subseteq$'' let $W\in\cZ(\cN_1)$. Then $W\in\cZ(\cM)$ and thus $W=\bigoplus_{i=1}^t W_i$
for some $W_i\in\cZ(\cM_i)$, which means in particular that $W_i\leq Z_i$.
But since $W\leq Z'_1=\bigoplus_{i=1}^t V_i$ and $V_i\leq Z_i$, we conclude $W_i\leq V_i$ for all~$i$.
This establishes \eqref{e-CFIdentity}.
\\
Now \eqref{e-Z'1Sum}--\eqref{e-CFIdentity} with \cref{T-Irred} tell us that~$\cN_1$ is reducible unless exactly one subspace~$V_i$ is nonzero.
In other words, irreducibility of~$\cN_1$ implies that, without loss of generality, $V_2=\ldots=V_t=0$ and thus $Z'_1=V_1\leq Z_1$.
\\
With the same argument we obtain $Z_1\leq Z'_j$ for some~$j\in\{1,\ldots,s\}$.
Hence $Z'_1\leq Z'_j$ and the directness of the sum $\bigoplus_{i=1}^s Z'_i$ implies $j=1$ and $Z'_1=Z_1$.
This shows $\cN_1=\cM|_{Z_1}=\cM_1$.
\\
Continuing in this way, we see that every summand~$\cN_j$ appears as a summand~$\cM_{l}$, and thus $s\leq t$. By symmetry $s=t$ and,
after reindexing, $\cM_i=\cN_i$ for all $i$.
\end{proof}

Now we can easily classify all \qM{}s on a 3-dimensional ground space.

\begin{exa}\label{E-ClassDim3}
The only irreducible \qM{}s on a 3-dimensional ground space are (up to equivalence) are $\cU_{1,3}$ and $\cU_{2,3}$.
Indeed, since any irreducible \qM{} on a 3-dimensional ground space is full, all 1-dimensional spaces must have rank~$1$ and all
$2$-dimensional spaces must have the same rank as the entire space.
With the aid of \cref{T-SplitOff} and the decomposition of trivial and free \qM{}s in \cref{E-TrivialFreeRed}(a) we obtain now all
\qM{}s (up to equivalence) on a 3-dimensional ground space. They are as follows. Note that this is independent of the field size~$q$.
\[
\begin{array}{|c||l|}
\hline
\rule[-.15cm]{0cm}{0.63cm}\text{rank\,}=0 & \cU_{0,1}\oplus\cU_{0,1}\oplus\cU_{0,1}\\ \hline
\rule[-.15cm]{0cm}{0.63cm}\text{rank\,}=1 & \cU_{0,1}\oplus\cU_{0,1}\oplus\cU_{1,1},\quad \cU_{0,1}\oplus\cU_{1,2},\quad \cU_{1,3}\\ \hline
\rule[-.15cm]{0cm}{0.63cm}\text{rank\,}=2 & \cU_{0,1}\oplus\cU_{1,1}\oplus\cU_{1,1},\quad\cU_{1,1}\oplus\cU_{1,2},\quad \cU_{2,3}\\ \hline
\rule[-.15cm]{0cm}{0.63cm}\text{rank\,}=3 & \cU_{1,1}\oplus\cU_{1,1}\oplus\cU_{1,1}\\\hline
\end{array}
\]
This classification has been derived earlier by different methods in \cite[Appendix A.4]{CeJu21}.
\end{exa}

On a 4-dimensional ground space the description of the irreducible \qM{}s is much harder.
With the same arguments as above we see that the uniform \qM{} $\cU_{1,4}$ is the only irreducible \qM{} of rank~$1$, and thus by duality
$\cU_{3,4}$ is the only one of rank~$3$.

In order to discuss \qM{}s of rank~$2$ we need the notion of a partial $k$-spread.
Recall that a subset~$\cV$ of $\cL(E)$ consisting of $k$-dimensional subspaces is called a \emph{partial $k$-spread} if
$V\cap W=0$ for all $V\neq W$ in~$\cV$.
It is called a \emph{$k$-spread} if $|\cV|=(q^n-1)/(q^k-1)$, where $\dim E=n$.
This is the maximum possible size of a partial $k$-spread and achievable if and only if~$k$ divides~$n$.
Hence if $k\mid n$, a partial $k$-spread of size~$s$ exists for all $1\leq s\leq (q^n-1)/(q^k-1)$.

\begin{prop}\label{P-Dim4}
Let $\dim E=4$.
\begin{alphalist}
\item Let $t\in\{2,\ldots,q^2+3\}\setminus\{4\}$ and~$\cV$ be a partial 2-spread of size $t-2$.
         On $\cL(E)$ define
         \begin{equation}\label{e-rhoSetV}
              \rho_{\cV}(V)=\left\{\begin{array}{cl} 1,&\text{if }V\in\cV,\\ \min\{2,\dim V\},&\text{otherwise.}\end{array}\right.
         \end{equation}
         Then $\cM_{\cV}=(E,\rho_{\cV})$ is an irreducible \qM{} with
         $\cZ(\cM_{\cV})=\cV\cup\{0,E\}$, thus $|\cZ(\cM)|=t$.
          Moreover, $Z_1\wedge Z_2= Z_1\cap Z_2$ and $Z_1\vee Z_2=Z_1+Z_2$ for all $Z_i\in\cZ(\cM)$.
\item  Each irreducible \qM{} of rank~$2$ on~$E$ is of the form $\cM_{\cV}=(E,\rho_{\cV})$ for some partial 2-spread~$\cV$.
         Its number of cyclic flats is $|\cV|+2$, which is at most $q^2+3$.
\item $\cU_2(E)$ is the unique irreducible \qM{} of rank~$2$ with exactly 2 cyclic flats.
\item Let $t\geq3$ and~$\cV,\,\cW$ be partial 2-spreads of size $t-2$ and $\cM_{\cV},\,\cM_{\cW}$ be the associated \qM{}s.
        Then there exists a bijection~$\alpha$ on~$\cL(E)$ such that $\rho_{\cV}(V)=\rho_{\cW}(\alpha(V))$ for all~$V\in\cL(E)$.
        In particular, there exists a rank-preserving and dimension-preserving lattice isomorphism between $\cZ(\cM_{\cV})$ and $\cZ(\cM_{\cW})$.
        For this reason we may call~$\cM_{\cV}$ and~$\cM_{\cW}$ ``bijectively equivalent''.
\end{alphalist}
\end{prop}

\begin{proof}
(a)
First of all, \cite[Prop.~4.6]{GLJ22Gen} tells us that $\cM_{\cV}:=(E,\rho_{\cV})$ is indeed a \qM{}.
One easily checks that  $\cZ(\cM_{\cV})=\cV\cup\{0,E\}$.
In particular, $|\cZ(\cM_{\cV})|=t$.
Since all spaces in~$\cV$ have dimension~2, it is obvious that $\cyc(Z_1\cap Z_2)=Z_1\cap Z_2$ and $\cl(Z_1+Z_2)=Z_1+Z_2$
for all $Z_1,Z_2\in\cZ(\cM)$.
For any $V\in\cV$ we have $\{Z\in\cZ(\cM_{\cV})\mid Z\leq V\}=\{0,V\}$, and therefore
\cref{T-Irred} shows that~$\cM_{\cV}$ is irreducible if and only if $|\cV|\neq 2$, which means $t\neq4$; for $\cV=\{V_1,V_2\}$ we obtain the reducible
\qM{} $\cM_{\cV}=\cM|_{V_1}\oplus\cM|_{V_2}\approx\cU_{1,2}\oplus\cU_{1,2}$.
\\
(b) Let $\cM=(E,\rho)$ be irreducible and of rank~$2$.
Set $\cV=\{V\in\cL(E)\mid \dim V=2,\,\rho(V)=1\}$. Using $\cl(0)=0$ and submodularity of~$\rho$ one obtains that~$\cV$
is a partial $2$-spread. Since $\rho(E)=2$, the set~$\cV$ is exactly the collection of subspaces~$V$ for which $\rho(V)\neq\min\{k,\dim V\}$, and
we conclude that  $\cM=\cM_{\cV}$. The rest follows from~(a) and the fact that any partial $2$-spread in~$E$ has size at most $q^2+1$.
\\
(c) An irreducible \qM{}~$\cM$ with exactly 2 cyclic flats must have $\cZ(\cM)=\{0,E\}$, and \cref{E-TrivialFreeRed}(c) establishes the result.
\\
(d)  Choose any dimension-preserving bijection~$\alpha$ on $\cL(E)$ such that $\alpha(\cV)=\cW$.
Then~$\alpha$ is also rank-preserving thanks to~\eqref{e-rhoSetV}.
\end{proof}

We have the following interesting example, where the bijection in \cref{P-Dim4}(d) is not induced by a linear (or semi-linear) isomorphism on~$E$.
Thus, we obtain non-equivalent \qM{}s whose lattices of flats are related by a rank-preserving lattice isomorphism.

\begin{exa}\label{E-Dim4}
This example is inspired by some facts from finite geometry.
Consider $\F=\F_3$.
The two sets~$\cA_1$ and~$\cA_2$ defined below are spread sets in $\F^2$, that is, subsets of $\F^{2\times 2}$ of order~$9$ such that $A-B\in\GL_2(\F)$ for all distinct $A,B$ in~$\cA_i$.
To any spread set one can associate a right quasifield. This is~$\F_9$ for $\cA_1$ and a right quasifield of Hall type for~$\cA_2$.
Since these two right quasifields are not isotopic we will obtain two $2$-spreads of the same size that are not related by a vector-space isomorphism.
This background is not needed since everything below can also be checked straightforwardly.
\\
In $\F^{2\times 2}$ let
\begin{align*}
  \cA_1&=\bigg\{\begin{pmatrix}0&0\\0&0\end{pmatrix},
       \begin{pmatrix}1&0\\0&1\end{pmatrix},\begin{pmatrix}1&1\\2&1\end{pmatrix},\begin{pmatrix}0&2\\1&0\end{pmatrix},
       \begin{pmatrix}1&2\\1&1\end{pmatrix}, \begin{pmatrix}2&0\\0&2\end{pmatrix},\begin{pmatrix}2&2\\1&2\end{pmatrix},
       \begin{pmatrix}0&1\\2&0\end{pmatrix},\begin{pmatrix}2&1\\2&2\end{pmatrix}\bigg\},\\[2ex]
    \cA_2&=\bigg\{\begin{pmatrix}0&0\\0&0\end{pmatrix},
       \begin{pmatrix}1&0\\0&1\end{pmatrix},\begin{pmatrix}2&0\\0&2\end{pmatrix},\begin{pmatrix}1&1\\1&2\end{pmatrix},
       \begin{pmatrix}2&2\\2&1\end{pmatrix}, \begin{pmatrix}0&1\\2&0\end{pmatrix},\begin{pmatrix}0&2\\1&0\end{pmatrix},
       \begin{pmatrix}1&2\\2&2\end{pmatrix},\begin{pmatrix}2&1\\1&1\end{pmatrix}\bigg\}.
\end{align*}
Note that $\cA_1$ is a subspace of~$\F^{2\times2}$, whereas~$\cA_2$ is not.
Define the set of matrices
\[
  \cB_i=\big\{\big(0\mid I\big)\big\}\cup \big\{\big( I\mid A\big)\mid A\in\cA_i\big\}\subseteq\F^{2\times4}\ \text{ for }\ i=1,2.
\]
Then
\[
   \cV_i=\big\{\rowsp(M)\mid M\in\cB_i\big\},\ i=1,2,
\]
are collections of 2-dimensional subspaces in $\cL(\F^4)$.
One easily verifies that $V\cap W=0$ for all distinct $V,W\in\cV_i$ (or argues that~$\cA_1$ and~$\cA_2$ are spread sets).
Hence~$\cV_1$ and~$\cV_2$ are $2$-spreads in~$\F^4$ because their cardinality is $10=q^2+1$.
By \cref{P-Dim4}, we obtain two irreducible \qM{}s $\cM_{\cV_i}=:\cM_i=(\F^4,\rho_i)$.
We have the following properties.
\\
(1) $\cM_1$ and $\cM_2$ are not equivalent in the sense of \cref{D-Equiv}.
Indeed, there exists no vector space automorphism~$\alpha$ on~$\F^4$ that maps the set~$\cV_1$ to~$\cV_2$.
This follows from the fact that the field $\F_9$ is not isotopic to any Hall quasifield or can be checked directly using a computer algebra
system by showing that no matrix in $\GL_4(\F)$ maps the set~$\cB_1$ to the set~$\cB_2$ (after taking the reduced row echelon forms).
Now it is clear that the \qM{}s $\cM_1$ and~$\cM_2$ are also not lattice-equivalent either because by the Fundamental Theorem of
Projective Geometry any lattice isomorphism is given by a semi-linear isomorphism, which over a prime field is linear.
\\
(2) $\cF(\cM_i)=\cO(\cM_i)=\cZ(\cM_i)=\cV_i\cup\{0,\F^4\}$, and in each of these lattices meet and join are simply intersection and sum.
Moreover, there exits a rank-preserving lattice isomorphism between
$\cF(\cM_1)$ and $\cF(\cM_2)$.
The identities follow from the fact that~$\cV_i$ is a $2$-spread (as opposed to just a partial spread), which means that every $1$-dimensional
space $\subspace{x}$ is contained in a subspace of~$\cV_i$ and therefore $\subspace{x}$ is not a flat.
The rest of the identities is easily verified, and the remaining statements follow from \cref{P-Dim4}.
\\
(3) It is remarkable that~$\cM_1$ is representable over~$\F_9$, whereas~$\cM_2$ is not representable over any field extension of~$\F$.
Indeed, one easily verifies that $\cM_1=\cM_G$, where
\[
     G=\begin{pmatrix}1&\omega^2&0&0\\0&0&1&\omega^2\end{pmatrix}\in\F_{9}^{2\times 4},
\]
where $\omega\in\F_9$ satisfies $\omega^2+2\omega+2=0$.
To see that~$\cM_2$ is not representable, suppose that $\cM_2=\cM_{G'}$ for some $G'\in\F_{3^m}^{2\times 4}$ (for some $m$).
Then $\rk G'=2$ and we may assume $G'$ to be in reduced row echelon form. Next, any such~$G'$ must satisfy
$\rk\big(G'\big(0\mid I\big)\T\big)=\rk\big(G'\big(I\mid 0\big)\T\big)=\rk\big(G'\big(I\mid I\big)\T\big)=1$
(see the set~$\cA_2$ and \cref{D-ReprG}), and therefore must be of the form
\[
    G'=\begin{pmatrix}1&a&0&0\\0&0&1&a\end{pmatrix}\text{ for some }a\in\F_{3^m}.
\]
Using now $\rk\big(G'(I\mid A)\T\big)=1$ for $A$ being the fourth and the sixth matrix in the set~$\cA_2$ leads to $a=1$.
But then the \qM{} $\cM_{G'}$ has loops, and thus~$\cM_{G'}\neq\cM_2$.
Thus~$\cM_2$ is not representable.
\end{exa}

The above discussion suggests that a classification of the irreducible \qM{}s on a 4-dimensional ground space appears to be quite challenging.

\bibliographystyle{abbrv}
\bibliography{literatureAK,literatureLZ}
\end{document}